\documentclass[final]{siamltex}
\usepackage{subcaption}
\usepackage{amsmath,amsfonts,amssymb}
\usepackage{graphicx}
\usepackage{multirow}
\usepackage{hyperref}
\usepackage{amsmath}
\usepackage{bm}
\usepackage{verbatim}
\usepackage{here}
\usepackage{caption}
\usepackage{moresize}
\usepackage[matha,mathx]{mathabx}
\usepackage{mathrsfs}
\usepackage{calligra}
\usepackage{calrsfs}
\usepackage{amsfonts}
\usepackage{amssymb}
\usepackage{graphicx}
\usepackage{tikz}
\usepackage{algorithm}
\usepackage{algorithmic}
\usepackage{graphicx}
\usepackage{amsfonts}    
\usepackage{lipsum}
\usepackage{systeme}
\usepackage{enumerate}

\usepackage{amsfonts}
\usepackage{graphicx}
\usepackage{epstopdf}
\usepackage{algorithmic}
\ifpdf
\DeclareGraphicsExtensions{.eps,.pdf,.png,.jpg}
\else
\DeclareGraphicsExtensions{.eps}
\fi

\title{Multidimensional extrapolated global  proximal gradient and applications for  image processing}

\author{A. Bentbib\thanks{Laboratory LAMAI, Faculty of Science and Technology, University Cadi Ayyad, Marrakech } \and K. Jbilou\thanks{Universit\'e de Lille Nord de France, LMPA, ULCO, 50 rue F. Buisson,  Calais-Cedex, France and Vanguard Center, university um6p, Benguérir Morocco} \and R. Tahiri\footnotemark[1]}

\makeatletter
\newcommand*{\addFileDependency}[1]{
	\typeout{(#1)}
	\@addtofilelist{#1}
	\IfFileExists{#1}{}{\typeout{No file #1.}}
}
\makeatother

\begin{document}

\maketitle

\begin{abstract}
The proximal gradient method is a generic technique introduced to  tackle the non-smoothness in optimization problems, wherein the objective function is expressed as the sum of a differentiable convex part and a non-differentiable regularization term. Such problems with tensor format are of interest in many fields of applied mathematics such as image and video processing. Our goal in this paper is  to address the solution of such problems with a more general form of the regularization term. An adapted iterative proximal gradient method is introduced for this purpose. Due to the slowness of the proposed algorithm, we use new tensor extrapolation methods to enhance its convergence. Numerical experiments on color image deblurring are conducted to illustrate the efficiency of our approach.
\end{abstract}



\section{Introduction}
\label{sect:intro}  
In this paper we are concerned with the  minimization problem

\begin{equation}\label{problem}
\min _{\mathcal{X} \in \Omega}\left(\frac{1}{2}\|\mathscr{F}(\mathcal{X})-\mathcal{B}\|_{F}^{2}+\varphi_{\mu}( \mathcal{X})\right),
\end{equation}
where the solution $\mathcal{X}$ and the observation $\mathcal{B}$ are $N$ th-order tensors in $\mathbb{R}^{I_{1} \times I_{2} \times \cdots \times I_{N}}$, $\mathscr{F}$ is a given linear tensor operator on $\mathbb{R}^{I_{1} \times I_{2} \times \cdots \times I_{N}}$, and  $\Omega$ is assumed to be a nonempty closed bounded convex set of the space $\mathbb{R}^{I_{1} \times I_{2} \times \cdots \times I_{N}}$. In this paper, the space of thensors that will be considered will be denoted by $\mathbb{T}=\mathbb{R}^{I_{1} \times I_{2} \times \cdots \times I_{N}}$.\\
The regularization term consists of the positive  real-valued function  $\varphi_{\mu}$  assumed to have the composition form
\begin{equation}\label{phi}
\begin{aligned}
\varphi_{\mu}=\mu\phi \circ  \mathscr{L} : &\mathbb{T}  \longrightarrow \mathbb{T}\longrightarrow \mathbb{R}_{+} \\
 \mathcal{X} & \overset{\mathscr{L}}{\longmapsto} \mathscr{L}(\mathcal{X}) \overset{\mu\phi}{\longmapsto }\mu\phi(\mathscr{L}(\mathcal{X})),
\end{aligned}
\end{equation}
where $\mathscr{L}$ is a given linear tensor operator, $\phi$ is a closed proper lower semicontinuous convex  non-differentiable function and $\mu$ is a positive regularization parameter. Under  these assumptions on $\varphi_{\mu}$ and $\Omega$, it was proven that  there exists a unique solution of the  problem (\ref{cnvx}); see  \cite{2bauschke,13combettes,15ekeland,21haddad,41rockafellar} for more details.
 
The convex optimization problem (\ref{problem}) provides a general framework  that will cover a wide range of high-order regularization problems. Specific choices of $\phi$ and $\mathscr{L}$ lead to various types of regularisation, including  $l_1$-regularization propblems  \cite{ista1,ista,ista2}, $l_1$-$l_2$-regularization process  \cite{l2,l1}, and total variation regularization problems  \cite{t2,benc,t1,42rudin} which  aim is preserving sharp edges (discontinuities) in data while removing noise and  unwanted fine-scale detail. Section \ref{S7} is dedicated to exploring specific instances of these cases.

The problem (\ref{problem}) represents a constrained  multidimensional  regularization problem that has  many  applications in different areas, such as  artificial neural network and machine learning  \cite{m2,m1}, image and video processing  \cite{benc,i2,i3,i1}. Our goal in this paper is to develop a general gradient-like method to handle the non-differentiability of the regularization term  in problem (\ref{problem}). Our minimization  approach is twofold. Firstly,   computing the proximal mapping of the regularization term, we look for a unconstrained minimizer of  the problem (\ref{problem}), that is a minimisation of (\ref{problem}) on the whole space $\mathbb{T}$. This minimizer will be determined by computing the proximal mapping of the regularisation term. Secondly, after compute the unconstrained minimizer, we project it  over the convex set $\Omega$  by applying  the Tseng’s splitting algorithm  \cite{2bauschke}.

 The framework of the proximal gradient algorithms, to which this work belongs, has been thoroughly investigated under various contexts  with an emphasis on establishing conditions under which these algorithms converge. The advantage of the proximal type methods is in their simplicity. However, they have also been recognized as  slow methods.  The numerical experiments conducted in this paper
provide further  grounds to that claim by showing that, under some choices of 
the operator $\mathscr{F} $  and function $\varphi_{\mu}$, the sequences  produced by the proposed algorithm shares a  very slow  rate of convergence. In order to overcome this drawback, we propose two  accelerated versions   by incorporating  very recent tensor extrapolation methods introduced in \cite{ridwane,e8,jbilou}

The rest of this  work is structured as follows. In  Section \ref{S2}, we provide a review of standard definitions and some useful properties. Moving on to Section \ref{S3}, we introduce a general multidimensional double proximal gradient method using the tensorial form of the objective function in the minimization problem (\ref{problem}). Section \ref{S4} is devoted to discuss some special cases of the problem (\ref{problem}).  In Section \ref{S5}, we suggest enhancing the proposed algorithm's speed by incorporating efficient tensorial extrapolation techniques. Section \ref{S6} is dedicated to applying the proposed algorithm to color image completion, and inpainting. Finally, in Section \ref{S7}, we present our conclusion.

\section{Preliminaries  and notation}
\label{S2}

Tensor algebra  \cite{11cichocki, kolda26, lee28} was developed  to handle higher-dimensional representation and analysis. Maintaining data or operations in their inherent multidimensional format enhances the capacity of systems to accumulate and retain extensive amounts of diverse data, thereby ensuring the accuracy of modeling.\\
Now, let us introduce  some  tensor algebra tools needed in this work. For this, we adopt  notations used  in  \cite{kolda26} .  Lowercase letters is used for presenting vectors, e.g., $x$, and uppercase letters are for matrices, e.g., $X$, while calligraphic letters is devoted for  tensors, e.g., $\mathcal{X}$. We denote by $\mathbb{T}$ the space $\mathbb{R}^{I_{1} \times I_{2} \times \cdots \times I_{N}}$ and  $\mathbb{T}^{N}:=\mathbb{T} \times \mathbb{T} \times \cdots \times \mathbb{T}$. Let $\mathcal{X} \in \mathbb{T}$ be an $N$ th-order tensor of size $I_{1} \times \cdots \times I_{N}$. which entries are denoted by $\mathcal{X}_{i_{1}, i_{2}, \ldots, i_{N}}$ or $\mathcal{X}\left(i_{1}, i_{2}, \ldots, i_{N}\right)$. Zero tensor  is denoted by $\mathcal{O}$ (all its entries equal to zero). We denote by $\mathbf{X}_{(\mathbf{n})}\in \mathbb{R}^{I_{n}\times (I_{1}I_{2}\cdots I_{n-1}I_{n+1}\cdots I_{N})}$  the  n-mode matrix of the tensor $\mathcal{X}$ and we denote by $X_{n}=\mathcal{X}\left(:,:, \ldots, :,n\right)$ the frontal slices  of $\mathcal{X}$  obtained by fixing the last index at $n$, $1\leq n\leq I_N$.  
\begin{definition}
   Let $k \geq 1$ be an integer. The $k$-norm and the infinity norm of the tensor $\mathcal{X}$ are defined by
 
\begin{equation}
    \begin{cases}
     \||\mathcal{X} \||_{k} =& \left(\sum_{i_{1}=1}^{I_{1}} \sum_{i_{2}=1}^{I_{2}} \ldots \sum_{i_{N}=1}^{I_{N}}\left|\mathcal{X}\left(i_{1}, i_{2}, \ldots, i_{N}\right)\right|^{k}\right)^{1 / k}\\
      \||\mathcal{X} \||_{\infty}=&\max _{\substack{1 \leq i_{j} \leq I_{j} \\ 1 \leq j \leq N}}\left|\mathcal{X}\left(i_{1}, i_{2}, \ldots, i_{N}\right)\right|. 
    \end{cases}       
\end{equation}

\end{definition}

\begin{definition}
    Let $\mathcal{X}, \mathcal{Y} \in \mathbb{R}^{I_{1} \times I_{2} \times \cdots \times I_{N}}$. The inner product of $\mathcal{X}$ and $\mathcal{Y}$   is defined by

$$
\langle\mathcal{X} \mid \mathcal{Y}\rangle=\sum_{i_{1}=1}^{I_{1}} \sum_{i_{2}=1}^{I_{2}} \ldots \sum_{i_{N}=1}^{I_{N}} \mathcal{X}\left(i_{1}, i_{2}, \ldots, i_{N}\right) \mathcal{Y}\left(i_{1}, i_{2}, \ldots, i_{N}\right).
$$
It follows immediately that the associated norm is  the Frobenius norm denoted as $\| \mathcal{X} \|_{F}$. 
\end{definition}
\begin{definition} \cite{liu2012tensor32}\label{nuclearnorm}
The tensor nuclear norm $\||\mathcal{X} \||_{*}$ of the  tensor $\mathcal{X}$ is  defined as the  sum of the singular values of the $n$-mode matricization $\mathbf{X}_{(\mathbf{n})}$ of the tensor $\mathcal{X}$ and given by 
$$\||\mathcal{X} \||_{*}=\sum_{n=1}^{N}\| \mathbf{X}_{(\mathbf{n})}\|_{*}, $$
\noindent where $\| \mathbf{X}_{(\mathbf{n})}\|_{*}=\sum_{i=1}^{I_n}\sigma_i(\mathbf{X}_{(\mathbf{n})})$ with $\sigma_i(\mathbf{X}_{(\mathbf{n})})$ stands for the $i$th singular value of the matrix $\mathbf{X}_{(\mathbf{n})}$. 
\end{definition}

\medskip
\noindent Now, Let us recall some optimization notions that will be used later. For more details, see  \cite{parikh40}. Let $g: \mathbb{T} \longrightarrow \mathbb{R} \cup\{\infty\}$ be a closed proper convex function .

\begin{definition}(convex conjugate function)
    The function  $g^{*}$ defined by
    $$
\begin{aligned}
    g^{*}: & \mathbb{T} \longrightarrow \mathbb{R} \\
\mathcal{X} & \longrightarrow g^{*}(\mathcal{X})=\sup _{\mathcal{Y}}\left(\langle\mathcal{Y} \mid \mathcal{\mathcal{X}}\rangle-g(\mathcal{Y})\right)
\end{aligned}$$
\end{definition} 
is called the the convex conjugate function of the function $g$.
\medskip
\begin{proposition} \cite{parikh40}
  Let  $\mathcal{Y} \in \mathbb{T}$, then there exists a unique minimizer $\mathcal{X}$ for the  problem
  $$\underset{\mathcal{X}}{\operatorname{min}}\left(g(\mathcal{X})+\frac{1}{2}\|\mathcal{X}-\mathcal{Y}\|^{2}\right).$$
\end{proposition}

\begin{definition} ( proximal operator)\
 The operator  $\operatorname{prox}_{g}$ defined by 
$$
\begin{aligned}
\operatorname{prox}_{g}: & \mathbb{T} \longrightarrow \mathbb{R} \\
\mathcal{Y} & \longrightarrow \operatorname{prox}_{g}(\mathcal{Y})=\underset{\mathcal{X}}{\operatorname{argmin}}\left(g(\mathcal{X})+\frac{1}{2}\|\mathcal{X}-\mathcal{Y}\|^{2}\right),
\end{aligned}
$$
is called the proximal operator of $g$.
\end{definition}

\begin{proposition} \cite{parikh40}
    Let  $\alpha>0$, then 
    $$
\operatorname{prox}_{\alpha g}(\mathcal{U})=\underset{\mathcal{X}}{\operatorname{argmin}}\left(g(\mathcal{X})+\frac{1}{2 \alpha}\|\mathcal{X}-\mathcal{U}\|^{2}\right).
$$ 
The operator $\operatorname{prox}_{\alpha g}$ is  called the proximal operator of $g$ with the parameter $\alpha$.
\end{proposition}

\section{Global tensorial double proximal gradient method }\label{S3}

The purpose of this paper is the resolution of the generalized constrained tensorial  minimization problem (\ref{problem}). Let us set
\begin{equation}\label{ff}
\begin{aligned}
f: & \mathbb{T} \longrightarrow \mathbb{R}_{+} \\
\mathcal{X} & \longrightarrow f(\mathcal{X})=\frac{1}{2}\|\mathscr{F}(\mathcal{X})-\mathcal{B}\|_{F}^{2}, 
\end{aligned}
\end{equation}
where $\mathscr{F}$ and $\mathcal{B}$   are those in (\ref{problem}). Then, the  problem (\ref{problem}) is expressed as 
\begin{equation}\label{cnvx}
\min _{\mathcal{X} \in \Omega}(f(\mathcal{X})+\varphi_{\mu}(\mathcal{X})).
\end{equation}
It is clear that  the function $f$ is differentiable, and its gradient is given by
$$
\nabla f(\mathcal{X})=\mathscr{F}^{T}(\mathscr{F}(\mathcal{X})-\mathcal{B}).
$$
\noindent 
The non-differentiability of $\phi$ renders the function $\varphi_{\mu}$ non-differentiable,  increasing the complexity of solving the problem.

\subsection{Global tensorial double proximal gradient algorithm}
In this subsection, we present a noteworthy extension of the gradient descent method tailored for addressing the general tensorial convex minimization problem (\ref{problem}). The literature has seen diverse adaptations of the gradient descent technique to tackle various minimization problems, including nonlinear minimization problems  \cite{13combettes}, fractional optimization problems  \cite{bouhamidi2018conditional}, among others. The proximal gradient method serves as a generalized version of the gradient descent method, particularly adept at handling non-differentiability in the cost function; see $ \cite{2bauschke,bauschke1,21haddad,parikh40}$. \\ First, we consider the  minimization problem

\begin{equation}\label{problem22}
\min _{\mathcal{X} \in \mathbb{T}}(f(\mathcal{X})+\varphi_{\mu}(\mathcal{X})).
\end{equation}
Assume that, at the iteration $k$, we have formed an iterate  tensor $\mathcal{X}_{k}$ approximating the solution to the unconstrained minimization problem (\ref{problem22}). The quadratic approximation of $f$   at the iterated tensor $\mathcal{X}_{k}$, with $\alpha_{k}>0$, is expressed as follows
$$
\Phi_{k}(\mathcal{X})=f\left(\mathcal{X}_{k}\right)+\left\langle\mathcal{X}-\mathcal{X}_{k} \mid \nabla f\left(\mathcal{X}_{k}\right)\right\rangle+\frac{1}{2 \alpha_{k}}\left\|\mathcal{X}-\mathcal{X}_{k}\right\|_{F}^{2},
$$
where  the step size $\left(\alpha_{k}\right)_{k}$  is a positive constant that  depends, as will be shown later, on the Lipschitz constant  of $\nabla f$.
Then,  at each step $k$,  we can approximate the problem (\ref{problem22}) by replacing $f$ by its approximate $\Phi_{k}$. We have then, the following minimization problem
\begin{equation*}
\min _{\mathcal{X}}\left(\varphi_{\mu}(\mathcal{X})+f\left(\mathcal{X}_{k}\right)+\left\langle\mathcal{X}-\mathcal{X}_{k} \mid \nabla f\left(\mathcal{X}_{k}\right)\right\rangle+\frac{1}{2 \alpha_{k}}\left\|\mathcal{X}-\mathcal{X}_{k}\right\|_{F}^{2}\right),
\end{equation*}
we add the constant quantity $\|\alpha_{k} \nabla f\left(\mathcal{X}_{k}\right)\|_{F}^{2}-f\left(\mathcal{X}_{k}\right)$ to the cost function, since that does not change the solution, we obtain
\begin{equation*}
\min _{\mathcal{X}}\left(\varphi_{\mu}(\mathcal{X})+\|\alpha_{k} \nabla f\left(\mathcal{X}_{k}\right)\|_{F}^{2}+\left\langle\mathcal{X}-\mathcal{X}_{k} \mid \nabla f\left(\mathcal{X}_{k}\right)\right\rangle+\frac{1}{2 \alpha_{k}}\left\|\mathcal{X}-\mathcal{X}_{k}\right\|_{F}^{2}\right),
\end{equation*}
leading to 
\begin{equation}\label{Eq4}
\min _{\mathcal{X}}\left(\varphi_{\mu}(\mathcal{X})+\frac{1}{2 \alpha_{k}}\left\|\mathcal{X}-\mathcal{X}_{k}+\alpha_{k} \nabla f\left(\mathcal{X}_{k}\right)\right\|_{F}^{2}\right),
\end{equation}
which has a unique minimizer $\mathcal{Z}_{k}$ defined by  
\begin{equation}\label{minimizer}
\mathcal{Z}_{k}=\operatorname{prox}_{\alpha_{k} \varphi_{\mu}}\left(\mathcal{X}_{k}-\alpha_{k} \nabla f\left(\mathcal{X}_{k}\right)\right), \quad\forall k \in \mathbb{N},
\end{equation}
where the operator $\operatorname{prox}_{\alpha_{k}\varphi_{\mu}} $ stands for the proximal operator of $\varphi_{\mu}$ with the parameter $\alpha_{k}$.Then, as we can see from the above relation ,  two crucial elements are necessary to compute $\mathcal{Z}_{k}$. The first element  is the  constant $\left(\alpha_{k}\right)_{k}$ which will be discussed in   Subsection \ref{sub1}. The  second element is  the  operator $\operatorname{prox}_{\alpha_{k} \varphi_{\mu}}$ which is given in the following proposition. 
\medskip 
\begin{proposition}
For all $\mathcal{Y} \in \mathbb{T}$ and $\alpha>0$, the proximal operator of $\alpha \varphi_{\mu}$ is given by
\begin{equation}\label{Eq3}
\mathcal{Z}=\operatorname{prox}_{\alpha \varphi_{\mu}}(\mathcal{Y})=\mathcal{Y}+\alpha \mathscr{L}^{T}\left(\mathscr{P}_{\ast}\right),
\end{equation}
where $\mathscr{L}^{T}$ is the transpose   of  $\mathscr{L}$ and $\mathscr{P}_{\ast}$ is an optimal solution of
\begin{equation}\label{prim}
    \min _{\mathcal{P}}\left(\phi_{\mu}^{*}(-\mathcal{P})+\left\langle\frac{\alpha}{2} \mathscr{L}\left(\mathscr{L}^{T} (\mathcal{P})\right)+\mathscr{L}(\mathcal{Y} )\mid \mathcal{P}\right\rangle\right),
\end{equation}
with $\phi_{\mu}^{*}$ being the conjugate function of $\phi_{\mu}= \mu\phi$.
\end{proposition} 
\begin{proof}
    Let $\mathcal{Y} \in  \mathbb{T}$ and $\alpha>0$, then  $\operatorname{prox}_{\alpha} \varphi_{\mu}(\mathcal{Y})$ is the unique solution of the problem
\[
\min _{\mathcal{U}}\left(\varphi_{\mu}(\mathcal{U})+\frac{1}{2 \alpha}\|\mathcal{U}-\mathcal{Y}\|_{F}^{2}\right).
\]
Using   (\ref{phi}) and the fact that  $\varphi_{\mu}(\mathcal{U})=\phi_{\mu}(\mathscr{ L}( \mathcal{U}))$, the above constrained  minimization problem can be reformulated as the following constrained one

\begin{equation}\label{cons}
\min _{\mathcal{U}, \mathcal{V}}\left(\phi_{\mu}(\mathcal{V})+\frac{1}{2 \alpha}\|\mathcal{U}-\mathcal{Y}\|_{F}^{2}\right) \text { subject to }  \mathcal{V}=\mathscr{ L}(\mathcal{U}).
\end{equation}
The associated Lagrangian operator  $\textbf{L}$   is provided  as
\[
\textbf{L}(\mathcal{U}, \mathcal{V}, \mathcal{P})=\phi_{\mu}(\mathcal{V})+\frac{1}{2 \alpha}\|\mathcal{U}-\mathcal{Y}\|_{F}^{2}+\langle\mathcal{P} \mid \mathcal{V}-\mathscr{ L}(\mathcal{U})\rangle,\; \mathcal{P} \in  \mathbb{T}.
\]
Consequently, as demonstrated in   \cite{troltzsch1985glowinski20}, the solution  of the problem (\ref{cons}) corresponds to the saddle point of $\textbf{L}$, representing the solution to the Lagrangian primal problem
\[
\min _{\mathcal{U}, \mathcal{V}} \max _{\mathcal{P}} \textbf{L}(\mathcal{U}, \mathcal{V}, \mathcal{P}).
\]
Due to the separability of the Lagrangian with respect to $\mathcal{U}$ and $\mathcal{V}$, we can interchange the min-max order according to the well known min-max theorem  \cite{15ekeland, 21haddad}. Consequently, the Lagrangian dual problem can be expressed as

\begin{equation}\label{to}
    \max _{\mathcal{P}}\left[\min _{\mathcal{U}}(\underbrace{\frac{1}{2 \alpha}\|\mathcal{U}-\mathcal{Y}\|_{F}^{2}-\langle\mathcal{P} \mid \mathscr{L}(\mathcal{U})\rangle}_{h(\mathcal{U})})+\min _{\mathcal{V}}\left(\phi_{\mu}(\mathcal{V})+\langle\mathcal{P} \mid \mathcal{V}\rangle\right)\right].
\end{equation}
On  one side, due to the convexity and differntiability of the cost function $h$ of the problem in $\mathcal{U}$, its  minimizer denoted by $\mathcal{U}_{\ast}$ is exactly the root of its gradient function. A simple calculation gives  $\mathcal{U}_{\ast}= \mathcal{Y}+\alpha \mathscr{L}^{T} (\mathcal{P})$ with the corresponding optimal value equal to
\[
h(\mathcal{U}_{\ast}) =\frac{1}{2 \alpha}\left\|\mathcal{U}_{\ast}-\mathcal{Y}\right\|_{F}^{2}-\left\langle\mathcal{P} \mid \mathscr{L}( \mathcal{U}_{\ast})\right\rangle 
 =-\left\langle\frac{\alpha}{2} \mathscr{L}(\mathscr{L}^{T} (\mathcal{P}))+\mathscr{L}(\mathcal{Y})\mid \mathcal{P}  \right\rangle .
\]
On the other side , the right minimization problem in $\mathcal{V}$ is a direct application of  the convex conjugate function of $\phi_{\mu}$. In fact,
$$
\min _{\mathcal{V}}\left(\phi_{\mu}(\mathcal{V})+\langle\mathcal{P} \mid \mathcal{V}\rangle\right)=-\max _{\mathcal{V}}\left(\langle-\mathcal{P} \mid \mathcal{V}\rangle-\phi_{\mu}(\mathcal{V})\right)\overset{def}{=}-\phi_{\mu}^{*}(-\mathcal{P}).
$$
 Substituting in (\ref{to}), we  obtain the following subsequent dual problem
$$
\max _{\mathcal{P}}\left[-\phi_{\mu}^{*}(-\mathcal{P})-\left\langle\frac{\alpha}{2} \mathscr{L}\left(\mathscr{L}^{T}( \mathcal{P})\right)+\mathscr{L}(\mathcal{Y}) \mid \mathcal{P}\right\rangle\right],
$$
which is equivalemnt to  the minimization problem (\ref{prim}).
\end{proof}

 \noindent We have shown that  $\operatorname{prox}_{\alpha \varphi_{\mu}}(\mathcal{Y})=\mathcal{Y}+\alpha \mathscr{L}^{T} (\mathscr{P}_{\ast})$, where $\mathscr{P}_{\ast}$ is a minimizer of the problem (\ref{prim}), then, at each step $k$,  to compute the proximal operator of  $\varphi_{\mu}$, we need to solve the minimization problem (\ref{prim}). That is, at the step $k$, if we set
\begin{equation}\label{Eq0}
     \mathcal{Y}_{k}:=\mathcal{X}_{k}-\alpha_{k} \nabla f\left(\mathcal{X}_{k}\right),
\end{equation}
then, we have to solve the problem
\begin{equation}\label{doubl}
    \min _{\mathcal{P}}\left(\phi_{\mu}^{*}(-\mathcal{P})+\left\langle\frac{\alpha_{k}}{2} \mathscr{L}\left(\mathscr{L}^{T} (\mathcal{P})\right)+\mathscr{L}( \mathcal{Y}_{k}) \mid \mathcal{P}\right\rangle\right).
\end{equation}
For the sake of clarity and simplicity, let us consider the operators $\mathscr{J}$ and $\mathscr{N}_{k}$ defined by

$$
\begin{array}{rl}
\mathscr{J}: & \mathbb{T}\longrightarrow \mathbb{R}_{+} \\
& \mathcal{P} \longrightarrow \mathscr{J}(\mathcal{P})=\phi_{\mu}^{*}(-\mathcal{P}), \\
\mathscr{N}_{k} & \mathbb{T} \longrightarrow \mathbb{R} \\
& \mathcal{P} \longrightarrow \mathscr{N}_{k}(\mathcal{P})=\left\langle\frac{\alpha_{k}}{2} \mathscr{L}\left(\mathscr{L}^{T}( \mathcal{P})\right)+\mathscr{L}(\mathcal{Y}_{k}) \mid \mathcal{P}\right\rangle .
\end{array}
$$
The problem (\ref{doubl}) is then formulated as
\begin{equation}\label{doubl1}
    \min _{\mathcal{P}}\left(\mathscr{N}_{k}(\mathcal{P})+\mathscr{J}(\mathcal{P})\right).
\end{equation}
Since  $\mathscr{N}_{k}$  is a closed proper convex  differentiable function, and the functional $\mathscr{J}$ is closed proper and convex, the problem (\ref{doubl1})  shares the same structure and conditions as the  problem (\ref{problem22}), where $\mathscr{N}_{k}$ and $\mathscr{J}$ play  roles analogous to $f$ and $\varphi_{\mu}$ respectively. Consequently, we use  again the same approach (proximal gradient) to address (\ref{doubl1}). Therefore, we can approximate  the solution of  (\ref{doubl1}) via the sequence $\left(\mathcal{P}_{l}\right)_{l}$ defined as
\begin{equation}
    \forall l \in \mathbb{N}, \quad \mathcal{P}_{l+1}=\operatorname{prox}_{\beta_{l} \mathscr{J}}\left(\mathcal{P}_{l}-\beta_{l} \nabla \mathscr{N}_{k}\left(\mathcal{P}_{l}\right)\right),
\end{equation}
with $\beta_{l}>0$ being a step size parameter that depends on the Lipshitz constant of the functional $\mathscr{N}_{k}$. Three crucial elementss are then  necessary to compute the above sequence. The first one is the parameter $\left(\alpha_{k}\right)_{k}$ which will be discussed in Subsection \ref{sub1}. The  second one is the gradient of the differentiable function $\mathscr{N}_{k}$, and finally the  operator $\operatorname{prox}_{\beta_{l}\mathscr{J}}$. For the gradient of $\mathscr{N}_{k}$, it is immediate to see that
\begin{equation}\label{Eq1}
\forall k, \quad \nabla \mathscr{N}_{k}(\mathcal{P})=\alpha_{k} \mathscr{L}\left(\mathscr{L}^{T}( \mathcal{P})\right)+\mathscr{L}(\mathcal{Y}_{k}) .
\end{equation}
\noindent For the proximal operator $\operatorname{prox}_{\beta_{l} \mathscr{J}}$,  let us first recall the following proposition relating the proximal mapping of any proper closed convex function $g$ by their conjugates.
\medskip 
\begin{proposition}  \cite{moreau1965proximite37}
 \begin{equation}\label{proxm}
     \operatorname{prox}_{g}(x)+\operatorname{prox}_{g^{*}}(x)=x, \quad \forall x .
 \end{equation}
\end{proposition}

\begin{proposition}
  
For all $\mathcal{P} \in \mathbb{T}$, the proximal mapping of $\beta_{l} \mathscr{J}$  can be expressed as 
\[
\operatorname{prox}_{\beta_{l} \mathscr{J}}(\mathcal{P})=\mathcal{P}+\operatorname{prox}_{\beta_{l} \phi_{\mu}}(-\mathcal{P}), \quad \forall l \in \mathbb{N}.
\]
\end{proposition} 
\begin{proof} 
 For any $\mathcal{P} \in \mathbb{T}$, we have
\[
\begin{aligned}
\operatorname{prox}_{\beta_{l} \mathscr{J}}(\mathcal{P}) & =\underset{\mathcal{W}}{\operatorname{argmin}}\left(\mathscr{J}(\mathcal{W})+\frac{1}{2 \beta_{l}}\|\mathcal{W}-\mathcal{P}\|_{F}^{2}\right), \\
& =\underset{\mathcal{W}}{\operatorname{argmin}}\left(\phi_{\mu}^{*}(-\mathcal{W})+\frac{1}{2 \beta_{l}}\|\mathcal{W}-\mathcal{P}\|_{F}^{2}\right), \\
& =-\underset{\mathcal{V}}{\operatorname{argmin}}\left(\phi_{\mu}^{*}(\mathcal{V})+\frac{1}{2 \beta_{l}}\|\mathcal{V}+\mathcal{P}\|_{F}^{2}\right), \\
& =-\operatorname{prox}_{\beta_{l} \phi_{\mu}^{*}}(-\mathcal{P}), \\
&=-\mathcal{P}-\operatorname{prox}_{\beta_{l}\phi_{\mu}}(\mathcal{-P}) \quad\quad(\text{according to (\ref{proxm}) }).
\end{aligned}
\]
\end{proof}

\noindent Furthermore, since $\phi_{\mu}=\mu \phi$, the proximal operator $\operatorname{prox}_{\beta_{l} \phi_{\mu}}$ is equal to the proximal operator of the function $(\beta_{l} \mu)\phi$. That is  $\operatorname{prox}_{\beta_{l} \phi_{\mu}}=\operatorname{prox}_{(\beta_{l} \mu)\phi}$ 
 which is  a direct result of the proximal operator of the function $\phi$.
 
\noindent Suppose now that the proximal operator of the function $\phi$ is known. At iteration $k$, the steps of the algorithm computing the approximate solution  $\mathcal{Z}_{k}$ in (\ref{minimizer}) of the   problem (\ref{Eq4}) can be  summarized as 
 illustrated in Algorithm \ref{alg1}.

\begin{algorithm}[h]
\caption{Approximate minimizer $\mathcal{Z}_{k}$.}\label{alg1}
\begin{algorithmic}[1]
\REQUIRE $l_k\in \mathbb{N}$, $\mathcal{X}_{k}$ ,$\mu$, $\alpha_k,  , \beta_1,..., \beta_{l_k}$

\STATE  Compute the tensors $\mathcal{Y}_{k}$  via (\ref{Eq0}).
\STATE Compute the operator $\mathscr{L}(\mathscr{N}_{k})$ via (\ref{Eq1}). 
 \FOR{$l=1:l_k$} 
\STATE $\mathcal{R}_l=\mathcal{P}_{l}-\beta_{l} \mathscr{L}( \mathscr{N}_{k})\left(\mathcal{P}_{l}\right)$,
\STATE   $ \mathcal{P}_{l+1}=\mathcal{R}_l+\operatorname{prox}_{\beta_{l} \mu\phi}(-\mathcal{R}_l)$, 
\ENDFOR
\STATE Compute $\mathcal{Z}_k$ by   $\mathcal{Z}_k=\mathcal{Y}_k+\alpha_k \mathscr{L}^{T}\left(\mathcal{P}_{l_k +1}\right)$.
\end{algorithmic}
\end{algorithm}

\noindent Notice  that the step size scalar $\left(\alpha_{k}\right)_{k}$ depends on the Lipschitz constant $L(\nabla f)$ while the scalars $\beta_{l}$, $l=1,\ldots,l_k$  are related to $L\left(\nabla \mathscr{N}_{k}\right)$. We will discuss their update process later.
\subsection{Approximate unconstrained solution via Tseng's splitting algorithm}
%
The question now is how about the solution of the constrained problem (\ref{cnvx})? In other words, how can we use   $\mathcal{Z}_{k}$ solution of the  problem  (\ref{problem2}) for getting a solution of the problem (\ref{cnvx}) under the
nonempty closed and convex constraint $\Omega$. The Tseng's splitting algorithm proposed in  \cite{2bauschke} is a useful and common tool that provides a response to our question. From the solution  $\mathcal{Z}_{k}$, this algorithm uses, at step $k$, the projection orthogonal onto the set 
$\Omega$ to produce an approximate solution $\mathcal{X}_{k+1}$ of $(\ref{cnvx})$ as follows 
\begin {equation}
\left\{\begin{array}{l}
\mathcal{Y}_{k} \quad \text{and} \quad \mathcal{Z}_{k} \quad \text{provided by Algorithm \ref{alg1} }, \\
\mathcal{Q}_{k}=\mathcal{Z}_{k}-\alpha_{k} \nabla f\left(\mathcal{Z}_{k}\right), \\
\mathcal{X}_{k+1}=\Pi_{\Omega}\left(\mathcal{X}_{k}-\mathcal{Y}_{k}+\mathcal{Q}_{k}\right).
\end{array}\right.
\end{equation}
where $\Pi_{\Omega}$ stands for the orthogonal projection onto $\Omega$.

\subsection{Selection of parameters $\alpha_{k}$ and $\beta_{l}$}\label{sub1}

Selecting the step size parameters $\alpha_{k}$ and $\beta_{l}$ is regarded as a standard condition that guarantees the convergence of the sequence $\left(\mathcal{X}_{k}\right)_{k}$ to the minimizer of the problem (\ref{cnvx}).  In   \cite{2bauschke,13combettes}, the authors provide sufficient conditions for this purpose. It is required that the positive values  $\alpha_{k}$ and $\beta_{l}$  should  be less than   the inverses of Lipschitz constants of the operators $\nabla f$ and $\nabla \mathscr{N}_{k}$, denoted by $L(\nabla f)$ and $L\left(\nabla \mathscr{N}_{k}\right)$,  respectively. That is, $0<\alpha_{k}<  \frac{1}{L(\nabla f)}$ and $0<\beta_{k}<\frac{1}{L\left(\nabla \mathscr{N}_{k}\right)}$.

\noindent For any pair $(\mathcal{X}, \mathcal{Y})$ in $\mathbb{T} \times \mathbb{T}$, we have

$$
\begin{aligned}
\|\nabla f(\mathcal{X})-\nabla f(\mathcal{Y})\|_{F} & =2\left\|\mathscr{F}^{T}(\mathscr{F}(\mathcal{X}))-\mathscr{F}^{T}(\mathscr{F}(\mathcal{Y}))\right\|_{F}, \\
& =2\left\|\mathscr{F}^{T}(\mathscr{F}(\mathcal{X}-\mathcal{Y}))\right\|_{F}, \\
& \leq 2\|| \mathscr{F}^{T} \circ \mathscr{F}\|| \cdot\| \mathcal{X}-\mathcal{Y} \|_{F},
\end{aligned}
$$
where $\circ$ refer to the composition operation. Then, we can choose the quantity $2\||\mathscr{F}^{T} \circ \mathscr{F} \||$ as a Lipschitz constant of the operator $\nabla f$.

\noindent As a result, the step size  $\alpha_{k}$ can be chosen as a stable value $\alpha_{k} \in$ $\left(0, \frac{1}{2\left|\|\mathscr{F}^{T} \circ \mathscr{F} \|\right|}\right)$.
In the scenario involving the operator $\nabla \mathscr{N}_{k}$ for a given step $k$, the Lipschitz constant $L\left(\nabla \mathscr{N}_{k}\right)$ is unknown. In such cases, the step sizes $\left(\beta_{l}\right)$ can be determined using a line search method  \cite{parikh40}. This entails applying the proximal gradient method with a straightforward backtracking step size rule, as follows
\begin{equation}\label{betal}
    \forall l \in \mathbb{N}, \quad \beta_{l}=\rho \beta_{l-1}.
\end{equation}
\noindent The main steps for the computation of the approximation  $\mathcal{X}_{k+1}$ of the problem (\ref{cnvx}) is given in the following algorithm.

\begin{algorithm}[h]
\caption{Global Tensorial  Proximal Gradient (GTPG)}\label{alg2}
\begin{algorithmic}[1]
\REQUIRE $\bar l\in \mathbb{N}$, $\mu,$ 
 $\mathcal{X}_{1}$, $\mathcal{P}_{1}$, $\nabla f$, $\operatorname{prox}_{\phi}$, $ , \beta_0,\rho$ and  $\alpha \in(0, \frac{1}{2\left|\|\mathscr{F}^{T} \circ \mathscr{F} \|\right|})$
\FOR{$k=1...$ until convergence} 
\STATE  Compute the tensor $\mathcal{Y}_{k}$  via (\ref{Eq0})
\STATE Compute the operator $\mathscr{L}(\mathscr{N}_{k})$ via (\ref{Eq1})
 \FOR{$l=1:\bar l$} 
 \STATE Update the  parameter $\beta_{l}$ via (\ref{betal})
\STATE $\mathcal{R}_l=\mathcal{P}_{l}-\beta_{l} \mathscr{L}( \mathscr{N}_{k}\left(\mathcal{P}_{l}\right))$
\STATE   $ \mathcal{P}_{l+1}=\mathcal{R}_l+\operatorname{prox}_{\beta_{l}\mu \phi}(-\mathcal{R}_l),$ 
\ENDFOR
\STATE Compute $\mathcal{Z}_k$ via:  $\mathcal{Z}_k=\mathcal{Y}_k+\alpha\mathscr{L}^{T}\left(\mathcal{P}_{l_k +1}\right)$
\STATE $\mathcal{Q}_{k}=\mathcal{Z}_{k}-\alpha \nabla f\left(\mathcal{Z}_{k}\right) $
\STATE $\mathcal{X}_{k+1}=\Pi_{\Omega}\left(\mathcal{X}_{k}-\mathcal{Y}_{k}+\mathcal{Q}_{k}\right)$
\ENDFOR
\end{algorithmic}
\end{algorithm}

\noindent The following theorem gives a convergence result on the consructed approximate sequence $\left( \mathcal{X}_{k}  \right )_k$ produced by Algorithm \ref{alg2}.
\medskip 
\begin{theorem}
    Let $\Omega$ be  a closed subset of $\mathbb{T}$. Then, the sequence $\left(\mathcal{X}_{k}\right)_{k}$ generated by Algorithm (\ref{alg2})  converges to  the unique solution, denoted by $\mathcal{X}_{*}$, of the problem (\ref{cnvx}) (or equivalently (\ref{problem}). That is, $\displaystyle \lim_{k \longrightarrow \infty}\left\|\mathcal{X}_{k}-\mathcal{X}_{*}\right\|_{F} =0$ .
\end{theorem}
\medskip
\begin{proof}
The functions $f$ and $\varphi_{\mu}$ are proper lower semicontinuous and convex, with $f$ being Gateau differentiable and uniformly convex on $\mathbb{T}$ and since the set $\Omega$ is a closed convex nonempty subset of $\mathbb{T}$, the strong convergence of the sequence $\left(\mathcal{X}_{k}\right)_{k}$  follows immediately from the general result  given in  \cite{2bauschke}.
\end{proof} 

\section{Special cases}\label{S4}
\subsection{Case $1$.  $\mathcal{L}=id_\mathbb{T}$ and $\phi =\|.\|_{1}$}
 The problem (\ref{problem}) becomes 
 \begin{equation}\label{problem1}
\min _{\mathcal{X} \in \Omega}\left(\frac{1}{2}\|\mathscr{F}(\mathcal{X})-\mathcal{B}\|_{F}^{2}+ \|\mathcal{X}\|_{1}\right),
\end{equation}
 where $\|\mathcal{X} \|_{1} = \sum_{i_{1}=1}^{I_{1}} \sum_{i_{2}=1}^{I_{2}} \ldots \sum_{i_{N}=1}^{I_{N}}\left|\mathcal{X}\left(i_{1}, i_{2}, \ldots, i_{N}\right)\right|$.
This problem (\ref{problem1}) is a direct extension, to tensor case, of the popular (vector) \textit{l$_1$-regularization} problems \cite{ista1,ista,ista2}. 
Furthermore, the tensorial proximal operator $\operatorname{prox}_{\beta_{l}\mu \phi}:=\operatorname{prox}_{\beta_{l} \mu }\|.\|_{1}$  can be seen as a direct extension of the soft thresholding operator  \cite{parikh40}. Then, similarly to  the vector case, $\operatorname{prox}_{\gamma } \|.\|_{1}$ with any $\gamma>0$, can be computed via the orthogonal projection on the $\| . \|_{\infty}$-unit ball; see  \cite{parikh40} for more details.  This leads to

$$
\left(\operatorname{prox}_{\gamma\|.\|_{1}.}(\mathcal{P})\right)_{i_{1}, \ldots, i_{N}}= \begin{cases}\mathcal{P}_{i_{1}, \ldots, i_{N}}-\gamma, & \mathcal{P}_{i_{1}, \ldots, i_{N}} \geq \gamma \\ 0, & \left|\mathcal{P}_{i_{1}, \ldots, i_{N}}\right|<\gamma \\ \mathcal{P}_{i_{1}, \ldots, i_{N}}+\gamma, & \mathcal{P}_{i_{1}, \ldots, i_{N}} \leq-\gamma\end{cases}
$$
Under these choices, the function $\mathscr{N}_k$, and hence its gradient  are given as
\[\mathscr{N}_{k}(\mathcal{P})=\left\langle\frac{\alpha_{k}}{2} \mathcal{P}+\mathcal{Y}_{k} \mid \mathcal{P}\right\rangle \quad \text{and} \quad \nabla \mathscr{N}_{k}(\mathcal{P})=\frac{\alpha_{k}}{2} \mathcal{P}+\mathcal{Y}_{k}. \]
\noindent  The  algorithm \ref{alg2}  is then recovered as a natural extension to tensor case of  the well known   \textit{iterative shrinkage-thresholding algorithms} (ISTA)  devoted to solve  vector \textit{l$_1$-regularization} problems, see  \cite{ista1,ista,ista2} . This leads to the following new algorithm  named Tensorial Iterative Shrinkage-Thresholding Algorithm (TISTA).
\begin{algorithm}[h]
\caption{Tensorial Iterative Shrinkage-Thresholding Algorithm (TISTA)}\label{alg3}
\begin{algorithmic}[1]
\REQUIRE $\bar l\in \mathbb{N}$, $\mu$
 $\mathcal{X}_{1}$, $\mathcal{P}_{1}$, $\nabla f$, $\operatorname{prox}_{\phi}$, $ \beta_0,\rho$  and  $\alpha \in(0, \frac{1}{2\left|\|\mathscr{F}^{T} \circ \mathscr{F} \|\right|})$,
\FOR{$k=1...$ until convergence} 
\STATE  $\mathcal{Y}_{k}=\mathcal{X}_{k}-\alpha \nabla f\left(\mathcal{X}_{k}\right)$,
 \FOR{$l=1:\bar l$} 
 \STATE Update the  parameter $\beta_{l}$ via (\ref{betal}):$\beta_{l}=\rho \beta_{l-1}$,
\STATE $\mathcal{R}_l=\mathcal{P}_{l}- \frac{\alpha\beta_{l}}{2} \mathcal{P}_{l}+\mathcal{Y}_{k} $,
\STATE   $ \mathcal{P}_{l+1}=\mathcal{R}_l+\operatorname{prox}_{\beta_{l}\mu\left\||\cdot\||_{1}\right.}(\mathcal{R}_l)$ .
\ENDFOR
\STATE Compute $\mathcal{Z}_k$ via:  $\mathcal{Z}_k=\mathcal{Y}_k+\alpha \mathcal{P}_{l_k +1}$,
\STATE $\mathcal{Q}_{k}=\mathcal{Z}_{k}-\alpha \nabla f\left(\mathcal{Z}_{k}\right) $,
\STATE $\mathcal{X}_{k+1}=\Pi_{\Omega}\left(\mathcal{X}_{k}-\mathcal{Y}_{k}+\mathcal{Q}_{k}\right)$.
\ENDFOR
\end{algorithmic}
\end{algorithm}

\noindent Assume that  the function $f$ in equation (\ref{ff}) is such that
\begin{equation}
\begin{aligned}
f: & \mathbb{T} \longrightarrow \mathbb{R}_{+} \\
\mathcal{X} & \longrightarrow f(\mathcal{X})=\frac{1}{2}\|\mathscr{A}*_N \mathcal{X}-\mathcal{B}\|_{F}^{2}, 
\end{aligned}
\end{equation}
where $\mathscr{A}\in \mathbb{T}^2$ is an $2Nth$-order tensor of size $I_{1} \times I_{2} \times \cdots \times I_{N}\times I_{1} \times I_{2} \times \cdots \times I_{N}$ and $*_N$ stands for the Einstein product operation  \cite{jbilou}, where $\mathscr{A}*_N \mathcal{X}$ is the  $I_{1} \times \ldots \times I_{N}$ tensor which entries are given by
	\begin{equation}\label{Einstien}
		\left(\mathcal{A} *_{N} \mathcal{X}\right)_{i_{1}, \ldots i_{N} }=\sum_{j_{1} \ldots j_{N}} \mathcal{A}_{i_{1}, \ldots i_{N} j_{1} \ldots j_{N}} \mathcal{X}_{j_{1} \ldots j_{N} }.
	\end{equation}
Through a simple calculation, we get  $\nabla f(\mathcal{X})=\mathscr{A}*_N \mathcal{X}-\mathcal{B} $. Substituting in Algorithm \ref{alg3}, we obtain a new algorithm the  called Einstein Tensorial Iterative Shrinkage-Thresholding Algorithm (EISTA) and is sunmmarized as follows.
\begin{equation}\label{pr}
\min _{\mathcal{X} \in \Omega}\left(\frac{1}{2}\|\mathscr{A}*_N \mathcal{X}-\mathcal{B}\|_{F}^{2}+ \||\mathcal{X}\||_{1}\right).
\end{equation}
\begin{algorithm}[h]
\caption{EISTA (Einstein product based ISTA)}\label{alg4}
\begin{algorithmic}[1]
\REQUIRE $\bar l\in \mathbb{N}$, $\mu$, $\mathcal{A}$, $\mathcal{B}$
 $\mathcal{X}_{1}$, $\mathcal{P}_{1}$, $\operatorname{prox}_{\phi}$, $ \beta_0,\rho$ and $\alpha \in(0, \frac{1}{2\|\mathscr{A}*_N \mathscr{A}\|})$.
\FOR{$k=1...$ until convergence} 
\STATE  $\mathcal{Y}_{k}=\mathcal{X}_{k}-\alpha(\mathscr{A}*_N \mathcal{X}_k-\mathcal{B})$.
 \FOR{$l=1:\bar l$} 
 \STATE Update the  parameter $\beta_{l}$ via (\ref{betal}):$\beta_{l}=\rho \beta_{l-1}$
\STATE $\mathcal{R}_l=\mathcal{P}_{l}- \frac{\alpha\beta_{l}}{2} \mathcal{P}_{l}+\mathcal{Y}_{k}$,
\STATE $\mathcal{P}_{l+1}=\mathcal{R}_l+\operatorname{prox}_{\beta_{l}\mu\left\||\cdot\||_{1}\right.}(\mathcal{R}_l)$.
\ENDFOR
\STATE Compute $\mathcal{Z}_k$ via:  $\mathcal{Z}_k=\mathcal{Y}_k+\alpha_k \mathcal{P}_{l_k +1}$,
\STATE $\mathcal{Q}_{k}=\mathcal{Z}_{k}-\alpha_{k} (\mathscr{A}*_N \mathcal{Z}_k-\mathcal{B}) $,
\STATE $\mathcal{X}_{k+1}=\Pi_{\Omega}\left(\mathcal{X}_{k}-\mathcal{Y}_{k}+\mathcal{Q}_{k}\right)$.
\ENDFOR
\end{algorithmic}
\end{algorithm}

\subsection{Case $2$ $\mathcal{L}=\nabla$ and $\phi =\||.\||_{1}$}
The problem (\ref{problem}) becomes
\begin{equation}\label{problem2}
\min_{\mathcal{X} \in \Omega}\left(\frac{1}{2}\|\mathscr{F}(\mathcal{X})-\mathcal{B}\|_{F}^{2}+\mu \||\nabla(\mathcal{X})\||_{1}\right),
\end{equation}
where the gradient operator $\nabla (\mathcal{X})$ of the  $N$-order tensor $\mathcal{X} \in \mathbb{T}$ is defined as a column block tensor in $\mathbb{T}^{N}$ consisting of the partial derivatives $\left(\nabla_{(n)} \mathcal{X}\right)_{n}$, i.e., $\nabla \mathcal{X}=$ $\left(\nabla_{(1)} \mathcal{X}, \ldots, \nabla_{(N)} \mathcal{X}\right)$, such that, for $n=1, \ldots, N$, the block tensor $\nabla_{(n)} \mathcal{X}$ is given by
$$
\left(\nabla_{(n)} \mathcal{X}\right)_{i_{1}, \ldots, i_{N}}= \begin{cases}\mathcal{X}_{i_{1}, \ldots, i_{n}+1, \ldots, i_{N}}-\mathcal{X}_{i_{1}, \ldots, i_{n}, \ldots, i_{N}} & \text { if } i_{n}<I_{n} \\ 0 & \text { if } i_{n}=I_{n}.\end{cases}
$$
The problem (\ref{problem2}) was considered in   \cite{benc}, resulting the  algorithm named  TDPG  and given as follows.
\begin{algorithm}[h]
\caption{[ \cite{benc}] Tensorial Double Proximal Gradient (TDPG) }\label{TDPG}
\begin{algorithmic}[1]
\REQUIRE $\bar l\in \mathbb{N}$, $\mu$
 $\mathcal{X}_{1}$, $\mathcal{P}_{1}$, $\nabla f$, $\operatorname{prox}_{\phi}$, $ \beta_0,\rho$ and $\alpha \in(0, \frac{1}{2\left|\|\mathscr{F}^{T} \circ \mathscr{F} \|\right|})$
\FOR{$k=1...$ until convergence} 
\STATE  Compute the tensors $\mathcal{Y}_{k}$  via (\ref{Eq0})
\STATE Compute the operator $\nabla\mathscr{N}_{k}$ via (\ref{Eq1})
 \FOR{$l=1:\bar l$} 
 \STATE Update the  parameter $\beta_{l}$ via (\ref{betal})
\STATE $\mathcal{R}_l=\mathcal{P}_{l}-\beta_{l} \nabla \mathscr{N}_{k}\left(\mathcal{P}_{l}\right)$
\STATE   $ \mathcal{P}_{l+1}=\mathcal{R}_l+\operatorname{prox}_{\beta_{l}\mu \phi}(-\mathcal{R}_l),$ 
\ENDFOR
\STATE Compute $\mathcal{Z}_k$ via:  $\mathcal{Z}_k=\mathcal{Y}_k+\alpha\nabla^{T}\left(\mathcal{P}_{l_k +1}\right)$
\STATE $\mathcal{Q}_{k}=\mathcal{Z}_{k}-\alpha \nabla f\left(\mathcal{Z}_{k}\right) $
\STATE $\mathcal{X}_{k+1}=\Pi_{\Omega}\left(\mathcal{X}_{k}-\mathcal{Y}_{k}+\mathcal{Q}_{k}\right)$
\ENDFOR
\end{algorithmic}
\end{algorithm}

\subsection{Case $3$. $\mathcal{L}=\mathcal{\nabla}$ and $\varphi_{\mu} =\mu\||\mathcal{L}( .)\||_{1}+\|| . \||_{*}$ } 
Considering these choices, the obtained  problem is as follows
\begin{equation}\label{problem23}
\min _{\mathcal{X} \in \mathbb{T}}\left(\frac{1}{2}\|\mathscr{F}(\mathcal{X})-\mathcal{B}\|_{F}^{2}+\mu \||\nabla(\mathcal{X})\||_{1}+\|| \mathcal{X}\||_{*}\right),
\end{equation}
where  $\|| . \||_{*}$ is the tensor nuclear norm given in Definition \ref{nuclearnorm}.
If,  in addition, we impose  that $\|| \mathcal{X}\||_{*}<\epsilon$ for some $\epsilon>0$ then, the unconstrained problem (\ref{problem23}) can be reformulated as the following constrained problem.
\begin{equation}\label{problem25}
\min _{\mathcal{X} \in \Omega}\left(\frac{1}{2}\|\mathscr{F}(\mathcal{X})-\mathcal{B}\|_{F}^{2}+\mu \||\nabla(\mathcal{X})\||_{1}\right)\quad \text{s.t}\quad \Omega=\{\mathcal{X} \in \mathbb{T}, \|| \mathcal{X}\||_{*}<\epsilon \}.
\end{equation}
The above problem has many applications such as in data completion problem and image processing. For instance, let us denote $\mathcal{M}\in \mathbb{T}$ our  incomplete data (for example, image with missing pixels). Let   $E$  be the set of the indices of the observed (available) elements, i.e.  $E=\{(i_1,i_2,...,i_N) ,  \mathcal{M}_{i_1,i_2,...,i_N} \text{ is observed}\}.$ Consider the  projection operator $\mathbf{P}_E: \mathbb{T}\longrightarrow\mathbb{T}$  given as  \begin{equation}\label{pix}    
(\mathbf{P}_E(\mathcal{X}))_{i_1,i_2,...,i_N} =\begin{cases}\mathcal{X}_{i_1,i_2,...,i_N} & \text { if } (i_1,i_2,...,i_N)\in E\\ 0 & \text { otherwise }.
\end{cases} 
\end{equation}
By taking $\mathscr{F}=\mathbf{P}_E$ and $\mathcal{B}=\mathbf{P}_E(\mathcal{M})$ in (\ref{problem25}), our proposed algorithm can be   applied to fill the missing  elements of the data $\mathcal{M}$. This will be illustrated in the numerical section.

\section{Convergence acceleration using  Extrapolation }\label{S5}
Extrapolation methods  \cite{e6,ridwane,e8,e1,e3,jbilou,j200,j100,e2} are very useful techniques for accelerating the convergence of slowly sequences as those  generated by our proposed algorithm (\ref{alg1}). In this section, we adopt two  recent Extrapolation methods, the  Global Tensor Topological Extrapolation Transformation (GT-TET) and the High-Order Singular Values Decomposition based on Minimal Extrapolation Method (HOSVD-MPE).

\subsection{GT-TET}

Topological  Extrapolation Algorithm (TEA) is one of the most popular extrapolation method to enhance convergence of slowly vector sequences. It is well-regarded to  its  theoretical clarity and numerical efficiency, particularly, when employed to address nonlinear problems, as exemplified by our primary problem (\ref{problem}). It was proposed for the first time by Brezenski \cite{brezi2} for vector sequences. Recently,  in  \cite{jbilou}, tensor version of TEA, namely GT-TET, was introduced to address multidimensional sequences.  \\
Let $(\mathcal{S}_n)_n$ be a convergent  sequence of  tensors, and $m \in \mathbb{N}^{*}$. Given  $2m+1$ terms $\mathcal{S}_n,\mathcal{S}_{n+1},...,\mathcal{S}_{n+2m}$, the GT-TET method provides an approximation $\mathcal{S}_{n,m}^{TET}$ to the limit as \begin{equation}\label{approxim}
  \mathcal{T}_{n,m}^{TET}= \mathcal{S}_n+\sum_{j=1}^{m}c_j^{(n)}\Delta\mathcal{S}_{n+j-1},
\end{equation}
where differences $\Delta\mathcal{S}_{n+j-1}=\mathcal{S}_{n+j}-\mathcal{S}_{n+j-1}$, and the weights $c_1^{(n)}, c_2^{(n)},...,c_m^{(n)}$ are determined via the solution of the least-squares problem:
\begin{equation}\label{matpro}
c^{(n,m)}=arg\min_{x\in \mathbb{R}^{m}}\| H^{(n,m)}x-b^{(n,m)}\|,
\end{equation}
where $c^{(n,m)}=\left[c_1^{(n)}, c_2^{(n)},...,c_m^{(n)}\right]^{T}\in \mathbb{R}^{m}$ and the matrix $H^{(n,m)}\in \mathbb{R}^{m\times m}$ and the vector $b^{(n,m)}\in \mathbb{R}^{m}$ are given as 
$$ H^{(n,m)} =
  \left[ {\begin{array}{cccc}
    \langle \mathcal{Y},\Delta^{2}\mathcal{S}_{n}\rangle  & \cdots & \langle \mathcal{Y},\Delta^{2}\mathcal{S}_{n+m-1}\rangle\\
    \langle \mathcal{Y},\Delta^{2}\mathcal{S}_{n+1}\rangle & \cdots & \langle \mathcal{Y},\Delta^{2}\mathcal{S}_{n+m}\rangle\\
    \vdots&\cdots&\vdots\\
    \langle \mathcal{Y},\Delta^{2}\mathcal{S}_{n+m-1}\rangle  & \cdots & \langle \mathcal{Y},\Delta^{2}\mathcal{S}_{n+2m-2}\rangle\\
  \end{array} } \right] , b^{(n,m)}= \quad \left[ {\begin{array}{c}
    \langle \mathcal{Y},\Delta\mathcal{S}_{n}\rangle \\
    \langle \mathcal{Y},\Delta\mathcal{S}_{n+1}\rangle \\
     \vdots\\
     \langle \mathcal{Y},\Delta\mathcal{S}_{n+m-1}\rangle\\
  \end{array} } \right] $$
where $\mathcal{Y} \in \mathbb{T}$ is  some chosen tensor, and the square deference $\Delta^{2}\mathcal{S}_{n+j-1}=\Delta\mathcal{S}_{n+j}-\Delta\mathcal{S}_{n+j-1}=\mathcal{S}_{n+j+1}-2\mathcal{S}_{n+j}+\mathcal{S}_{n+j-1}$.

\noindent Using the standard $QR$ decomposition of the matrix $H^{(n,m)}$, that is  $H^{(n,m)}=QR$ with $Q$ orthogonal and $R$ upper triangular, the solution of the problem (\ref{matpro}) (assuming full rank of $H^{(n,m)}$) is obtained via $$c^{(n,m)}= R^{-1}y \quad \text{such that }\quad y =Q^{T} b^{(n,m)}.$$ 
We  summarized the   steps of GT-TET in the  following algorithm.

\begin{algorithm}[h]
\caption{GT-TET}\label{alg6}
\begin{algorithmic}[1] 
\REQUIRE $\mathcal{S}_n,\mathcal{X}_{n+1},...,\mathcal{S}_{n+2m}$,
\ENSURE  $\mathcal{T}_{n,m}^{TET}$.
\STATE Compute the differences $\Delta\mathcal{S}_{n+j-1}, j=1,\cdots,2m-1$.
\STATE Compute the square differences $\Delta^2\mathcal{X}_{n+j-1}, j=1,\cdots,2m-2$.
\STATE  Construct the matrix $H^{(n,m)}$ and vector  $b^{(n,m)}$.
 \STATE Compute the $QR$ of $H^{(n,m)}$: $H^{(n,m)}=QR$. 
\STATE $y=Q^{T} b^{(n,m)}$, 
\STATE  $c^{(n,m)}= R^{-1}y$.

\STATE Compute  the approximation $\mathcal{T}_{n,m}^{TET}$ as given  in (\ref{approxim}). 

\end{algorithmic}
\end{algorithm}

\subsection{HOSVD-MPE}
The HOSVD-MPE is a polynomial extrapolation  method  recently proposed in  \cite{ridwane} to accelerate the convergence of tensor sequences.   It provides an approximation 
$\mathcal{X}_{n,m}^{H-M}$ to the limit as \begin{equation}\label{approxim2}
  \mathcal{X}_{n,m}^{H-M}= \sum_{j=1}^{m}c_j^{(n)}\mathcal{X}_{n+j-1},
\end{equation}
the weights $c_1^{(n,m)}, c_2^{(n,m)},...,c_{m}^{(n,m)}$ are determined by 
\begin{equation}\label{MOL}
c_j^{(n)}=\dfrac{\delta^{(n)}_{j}}{\sum_{i=1}^{m}\delta^{(n)}_{i}}\hspace{0.7cm} for\hspace{0.3cm} j=1,\ldots,m.
\end{equation}
where  $\delta^{(n,m)} = (\delta^{(n)}_{1},\delta^{(n)}_{1},\ldots,\delta^{(n)}_{m})^T $ is the solution of the    constrained least-squares tensor  problem
 \begin{equation}\label{mol}
\underset{\underset{\parallel x\parallel_2=1}{x\in \mathbb{R}^{m}}}{\min} \parallel\mathbb{ D}_{m}^{(n)}\bar{\times}_{m} x\parallel_{F},
  \end{equation}
with  $
\mathbb{D}_{m}^{(n)} =\left[ \Delta\mathcal{X}_{n},\Delta\mathcal{X}_{n+1},\ldots,\Delta\mathcal{X}_{n+m-1}\right] \in \mathbb{T}^{m}
$
 such that the $i^{th}$ frontal slice, obtained by fixing the last index at $i$, is given by  $\left[ \mathbb{D}^{(n)}_{m}\right]_{:,:,\ldots,:,i}=\Delta\mathcal{X}_{n+i-1}$, $1\leq i\leq m$.
In \cite{ridwane}, we  proposed the following algorithm  to compute the solution  the desired approximation $\mathcal{X}_{n,m}^{H-M}$ of the problem  (\ref{approxim2}).
 \begin{algorithm}[H]
  \begin{algorithmic}[1]
\caption{HOSVD-MPE\label{hosvdmp10}}
\REQUIRE $\mathcal{X}_{n},\mathcal{X}_{n+1},\ldots,\mathcal{X}_{n+m}.$
 \ENSURE $\mathcal{X}_{n,m}^{H-M}$
\STATE Compute the tensors $\Delta\mathcal{X}_{n},\Delta\mathcal{X}_{n+1},\ldots,\Delta\mathcal{X}_{n+m-1}$ 
\STATE Form the tensor $\mathbb{D}_{m}^{(n)}=\left[ \Delta\mathcal{X}_{n},\Delta\mathcal{X}_{n+1},\ldots,\Delta\mathcal{X}_{n+m-1}\right] .$
\STATE Compute the $m$-unfolding matrix of $\mathbb{ D}_{m}^{(n)}$ :$N=(\mathbb{ D}_{m}^{(n)})_{(m)}.$
\STATE Compute the square matrix $M:=NN^{T}\in \mathbb{R}^{m\times m} $.
  \STATE Compute the Eigen Value Decomposition (EVD) of $M:\hspace{0.5cm}M=V\Sigma V^{T}.$
 \STATE Determine $\delta^{(n,m)}= Ve_{m} $ with $e_{m}=(0,0,\ldots,1)^{T}\in \mathbb{R}^{m}$.
\STATE Determine $c_1^{(n)},c_1^{(n)},...,c_m^{(n)}$ via (\ref{MOL}).
\STATE  Set $\mathcal{X}_{n,m}^{H-M}= \sum_{j=1}^{m}c_j^{(n)}\mathcal{X}_{n+j-1}.$
\end{algorithmic}
\end{algorithm}

\subsection{Accelerated version of GTPG algorithm}
There are several schemes to apply extrapolation methods to a given sequence. In this work we adopt the restarting mode for integrating  the GT-TET and HOSVD-MPE methods  to  Algorithm \ref{alg2} which is the main algorithm of this work. Algorithms \ref{hos0} and Algorithm \ref{hos1} describe the  accelerated versions of GTPG using HOSVD-MPE and GT-TET, respectively. We denote by $(\mathcal{T}_{k})_k$ the extrapolated sequence produced by these two Algorithms \ref{hos0} and \ref{hos1}.

\begin{algorithm}[h]
  \begin{algorithmic}[1]
\caption{Accelerated GTPG using HOSVD-MPE (GTPG-HM)\label{hos0}}
\REQUIRE $m\in \mathbb{N}^{*}$, $\mathcal{T}_{1}=\mathcal{X}_{1}$ and $\epsilon$ (stop threshold).
\ENSURE $\mathcal{T}_{k}$.
\STATE $k=2,$
\STATE Starting from $\mathcal{X}_{1}$, compute  $\mathcal{X}_{1},\mathcal{X}_{2},\ldots,\mathcal{X}_{m}$ via Algorithm \ref{alg2} (GTPG), 
\STATE Compute $\mathcal{X}_{1,m}^{H-M}$ by HOSVD-MPE algorithm \ref{hosvdmp10}, 
\STATE Set $\mathcal{T}_{k}= \mathcal{X}_{1,m}^{H-M}$,

\IF{$ \frac{\parallel\mathcal{T}_{k}-\mathcal{T}_{k-1}\parallel}{\parallel\mathcal{T}_{k-1}\parallel}<\epsilon $ } 
  \STATE  Stop
\ELSE
    
   \STATE Set $\mathcal{X}_{1}=\mathcal{T}_{k}$ and $k=k+1$ and go to the second step.
\ENDIF
\end{algorithmic}
\end{algorithm}
\begin{algorithm}[h]
  \begin{algorithmic}[1]
\caption{Accelerated GTPG using GT-TET (GTPG-TET)\label{hos1}}
\REQUIRE $m,\in \mathbb{N}^{*}$, $\mathcal{T}_{1}=\mathcal{X}_{1}$ and $\epsilon$ (stop threshold).
\ENSURE $\mathcal{T}_{k}$.
\STATE $k=2,$
\STATE Starting from $\mathcal{X}_{1}$, compute  $\mathcal{X}_{1},\mathcal{X}_{2},\ldots,\mathcal{X}_{2m+1}$ via Algorithm \ref{alg2} (GTPG), 
\STATE Compute $\mathcal{X}_{1,m}^{TET}$ by GT-TET algorithm \ref{alg6}, 
\STATE Set $\mathcal{T}_{k}= \mathcal{X}_{1,m}^{TET},$
\IF{$ \frac{\parallel\mathcal{T}_{k}-\mathcal{T}_{k-1}\parallel}{\parallel\mathcal{T}_{k-1}\parallel}<\epsilon $ } 
  \STATE  Stop
\ELSE  
   \STATE Set $\mathcal{X}_{1}=\mathcal{T}_{k}$ and $k=k+1$ and go to the second step.
\ENDIF
\end{algorithmic}
\end{algorithm}

\section{Numerical experiment}\label{S6}
In this section, we illustrated the efficiency of the proposed extrapolated  algorithm  TISTA-TET, TISTA-HM, TDPG-TET and TDPG-HM (see Table \ref{Tabalg}. As explained in  Section \ref{S4}, some image deblurring problems, such as image completion can be modeled within the context of our primary problem (\ref{problem}) and then  the applicability of our proposed method  to this kind of problems. The objective of image completion problems is to reconstruct missing regions within an observed image by using content-aware information extracted from undisturbed regions, based on the available pixels. Let  $\mathcal{X}_{orig} \in \mathbb{R}^{I_1\times I_2 \times 3}$ denote our original color image, and we uniformally/randomly mask off a portion of its entries that we regard  as missing values. Let   $\mathcal{B}\in \mathbb{R}^{I_1\times I_2 \times 3}$ the obtained  incomplete image and  let be $E$ the set of the indices of the available elements, i.e.  $E=\{(i_1,i_2,...,i_3) ,  \mathcal{B}_{i_1,i_2,i_3} \text{ is observed}\}$. We test the performance of our algorithms to recover the missing entries via the solution of the following problems
$$\min _{\mathcal{X} \in \Omega}\left(\frac{1}{2}\|\mathbf{P}_E(\mathcal{X})-\mathcal{B}\|_{F}^{2}+ \||\mathcal{X}\||_{1}\right), \hspace{1cm}\text{using the lgorithm TISTA \ref{alg3}}$$
$$\min _{\mathcal{X} \in \Omega}\left(\frac{1}{2}\|\mathbf{P}_E(\mathcal{X})-\mathcal{B}\|_{F}^{2}+\mu \||\nabla(\mathcal{X})\||_{1}\right),  \hspace{0.7cm}\text{using the algorithm TDPG \ref{TDPG}},$$
where $\Omega=\{\mathcal{X} \in \mathbb{R}^{I_1\times I_2 \times 3}, \; \text{such that}\; \|| \mathcal{X}\||_{*}<\epsilon \}$ and $\mathbf{P}_E$ is the projection operateur given in (\ref{pix}).
To assess the efficiency of our algorithm, we measure its performance using the peak signal-to-noise ratio ($PSNR$) using the standard \textit{psnr} function from Matlab,   and the relative error $Re$  defined as $Re=\dfrac{\|\mathcal{X}_{prx}-\mathcal{X}_{orig}\|}{\|\mathcal{X}_{orig}\|}$
where $\mathcal{X}_{prx}$ is the approximate solution. The  iterations will be stopped when  $\dfrac{\|\mathcal{X}_{k+1}-\mathcal{X}_{k}\|}{\|\mathcal{X}_{k}\|}<10^{-3}.$
 All computations were conducted using  MATLAB R2023a  on an 12th Gen Intel(R) Core(TM) i7-1255U  1.70 GHz  computer with 16 GB of RAM. We tested and compared the algorithms listed in the following table

\begin{table}[H]
\caption{Algorithms}\label{Tabalg}
\centering
\begin{tabular}{l|l}
  \hline
   Acronym & meaning\\ 
   \hline
 TISTA& basic algorithm TISTA \ref{alg3} \\
 TISTA-TET&extrapolated TISTA using topological extrapolation GT-TET  \\
 TISTA-HM& extrapolated TISTA using HOSVD-MPE extrapolation \\
 TDPG&basic algorithm TDPG \ref{TDPG}\\
 TDPG-TET&extrapolated TDPG using topological extrapolation GT-TET\\
 TDPG-HM&extrapolated TISTA using HOSVD-MPE extrapolation\\
     \hline 
\end{tabular}
\end{table}

\noindent In the tow subsequent parts, we will test the performance of our algorithms both for the randomly uncompleted images and uniformly uncompleted ones. Let us first investigate  the
efficiency of our proposed methods  in the case of  randomly missing pixels.

\subsection{Random missing pixels}
In this part, we use  the $250\times 250\times 3$  color images 'footbal.png', 'Peppers.bmp' and 'flamingos.jpg', and we randomly
mask off about $55\%$ of their  entries as shown in Figure \ref{Figurer1}. The corresponding results are shown in Figure \ref{figur2r}. In one hand, we see that TISTA and TDPG algorithms provide images with nearly the same clarity. In the other side, we can see clearly that the extrapolated versions of TISTA and  TDPG provide clearer images.\\
Figures \ref{figure3r} and \ref{figure3r1} illustrate,  the relative error and the PSNR curves produced by the six algorithms. We can see clearly that the extrapolated algorithms TISTA-TET, TISTA-HM and TDPG-TET, TDPG-HM   converge faster than the basic ones TISTA and TDPG.
  The curves in Figure \ref{figure3r2},  represent the convergence rates of the four  extrapolated methods TISTA-TET, TISTA-HM and TDPG-TET, TDPG-HM . The curves depict the efficiency of the four methods, highlighting the superiority of TISTA-TET and TISTA-HM in comparison to the TDPG-TET and TDPG-HM algorithms.
\begin{figure}[H]
          \centering
    \includegraphics[width=0.3\textwidth]{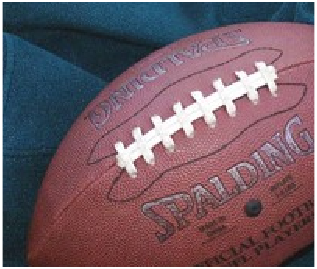}
    \includegraphics[width=0.3\textwidth]{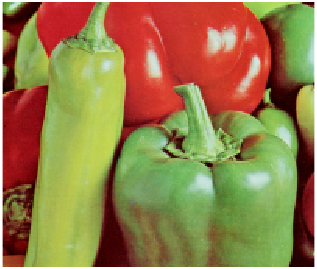}   
    \includegraphics[width=0.3\textwidth]{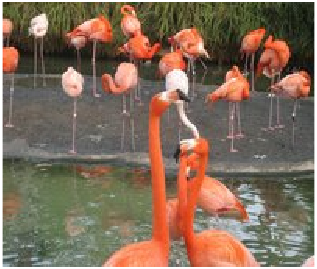}      \\
         \centering
    \includegraphics[width=0.3\textwidth]{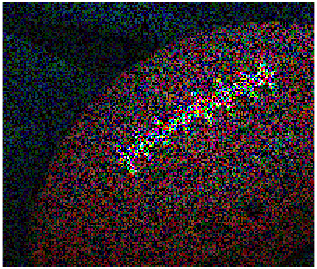}
    \includegraphics[width=0.3\textwidth]{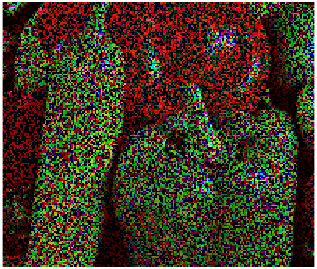}   
    \includegraphics[width=0.3\textwidth]{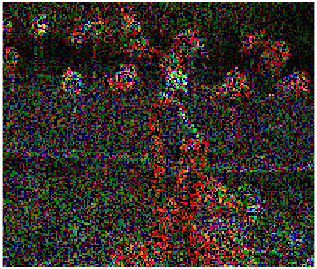}      
    \caption{The original images (1st row), the (randomly) $55\%$ incompleted images (2nd row), }\label{Figurer1}
\end{figure}

\begin{figure}[H]
 
\begin{minipage}{1.4cm}
        \centering
\caption*{\centering \ssmall{\ssmall TISTA}}
    \includegraphics[width=1.3\textwidth]{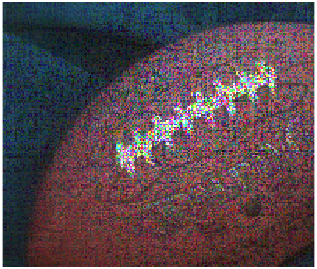}
\end{minipage}\hfill
\begin{minipage}{1.4cm}
        \centering
\caption*{\centering \ssmall{\ssmall TISTA-TET}}
    \includegraphics[width=1.3\textwidth]{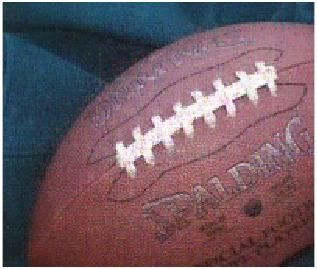}   
\end{minipage}\hfill
\begin{minipage}{1.4cm}
        \centering
\caption*{\centering \ssmall{\ssmall TISTA-HM}}
    \includegraphics[width=1.3\textwidth]{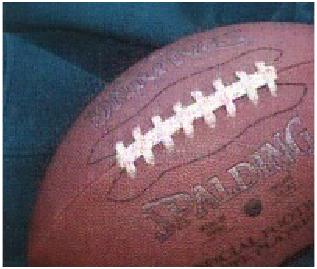}   
\end{minipage}\hfill
\begin{minipage}{1.4cm}
\caption*{\centering \ssmall{\ssmall TDPG}}
        \centering
    \includegraphics[width=1.3\textwidth]{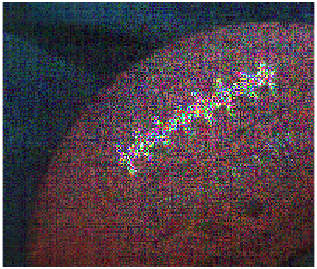}      
\end{minipage}\hfill
    \begin{minipage}{1.4cm}
        \centering
\caption*{\centering \ssmall{\ssmall TDPG-TET}}
    \includegraphics[width=1.3\textwidth]{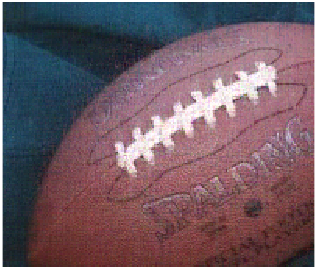}       
    \end{minipage}\hfill
    \begin{minipage}{1.4cm}
        \centering
\caption*{\centering \ssmall{\ssmall TDPG-HM}}
    \includegraphics[width=1.3\textwidth]{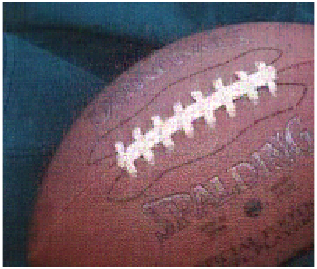}
    \end{minipage}\\
\centering

\begin{minipage}{1.4cm}
        \centering
    \includegraphics[width=1.3\textwidth]{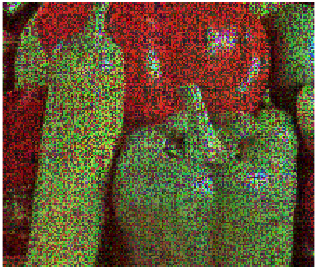}
\end{minipage}\hfill
\begin{minipage}{1.4cm}
        \centering
    \includegraphics[width=1.3\textwidth]{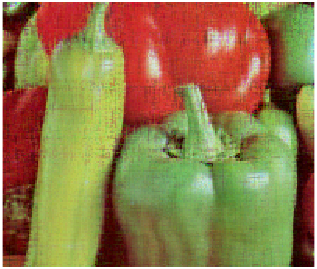} 
\end{minipage}\hfill
\begin{minipage}{1.4cm}
        \centering
    \includegraphics[width=1.3\textwidth]{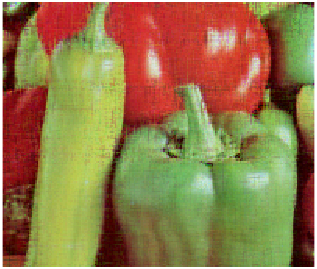} 
\end{minipage}\hfill
\begin{minipage}{1.4cm}
        \centering
    \includegraphics[width=1.3\textwidth]{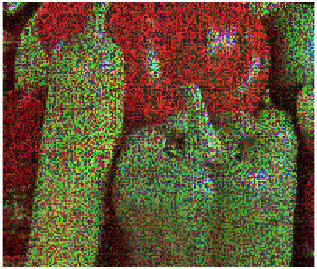}      
\end{minipage}\hfill
    \begin{minipage}{1.4cm}
        \centering
    \includegraphics[width=1.3\textwidth]{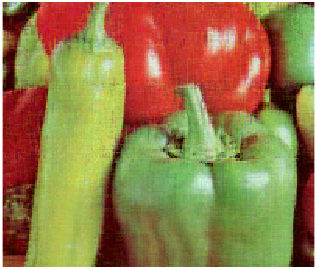}       
    \end{minipage}\hfill
    \begin{minipage}{1.4cm}
        \centering
    \includegraphics[width=1.3\textwidth]{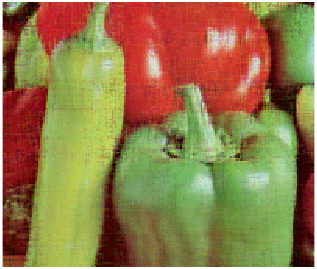}
    \end{minipage}\\
\begin{minipage}{1.4cm}
        \centering
    \includegraphics[width=1.3\textwidth]{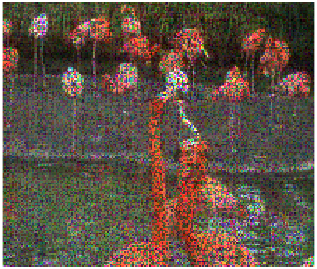}
    \end{minipage}\hfill
    \begin{minipage}{1.4cm}
        \centering
    \includegraphics[width=1.3\textwidth]{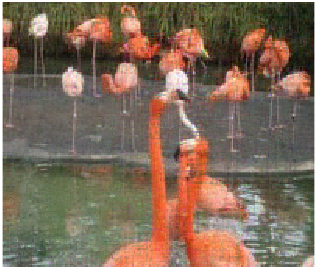}
    \end{minipage} \hfill
    \begin{minipage}{1.4cm}
        \centering
    \includegraphics[width=1.3\textwidth]{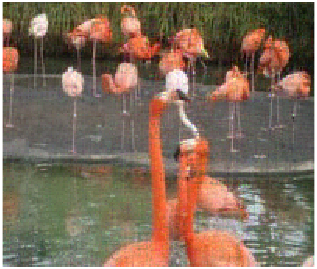}
    \end{minipage}\hfill
    \begin{minipage}{1.4cm}
        \centering      
        \includegraphics[width=1.3\textwidth]{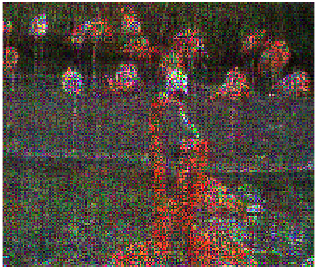}      
    \end{minipage}\hfill
    \begin{minipage}{1.4cm}
        \centering      
        \includegraphics[width=1.3\textwidth]{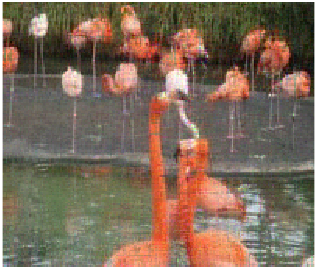}       
    \end{minipage}\hfill
    \begin{minipage}{1.4cm}
        \centering       \includegraphics[width=1.3\textwidth]{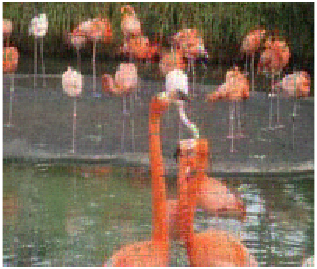}
    \end{minipage} 
    \caption{The recovered images by: TISTA (1th column), TISTA-TET (2th column),  TISTA-HM (3th column),  TDPG (4th column),  TDPG-TET (5th column) and b TDPG-HM (6th column).}\label{figur2r}
\end{figure}
\begin{figure}
 \begin{minipage}{5.8cm}
   \includegraphics[width=1.4\textwidth]{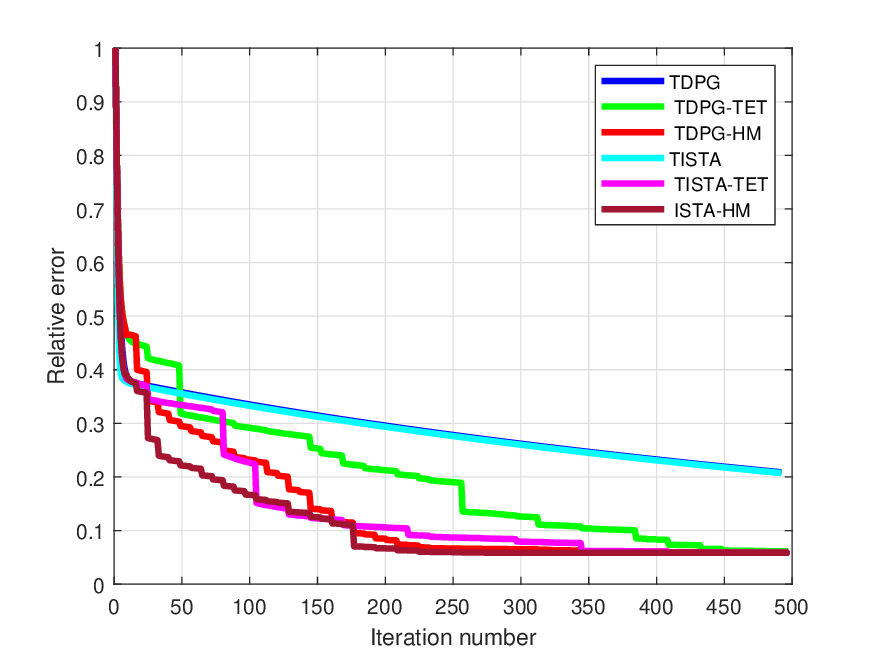}
    \caption{ The curves of relative errorr.}
    \label{figure3r}
     \end{minipage}\hspace{1.7cm}
\begin{minipage}{5.8cm}
   \includegraphics[width=1.4\textwidth]{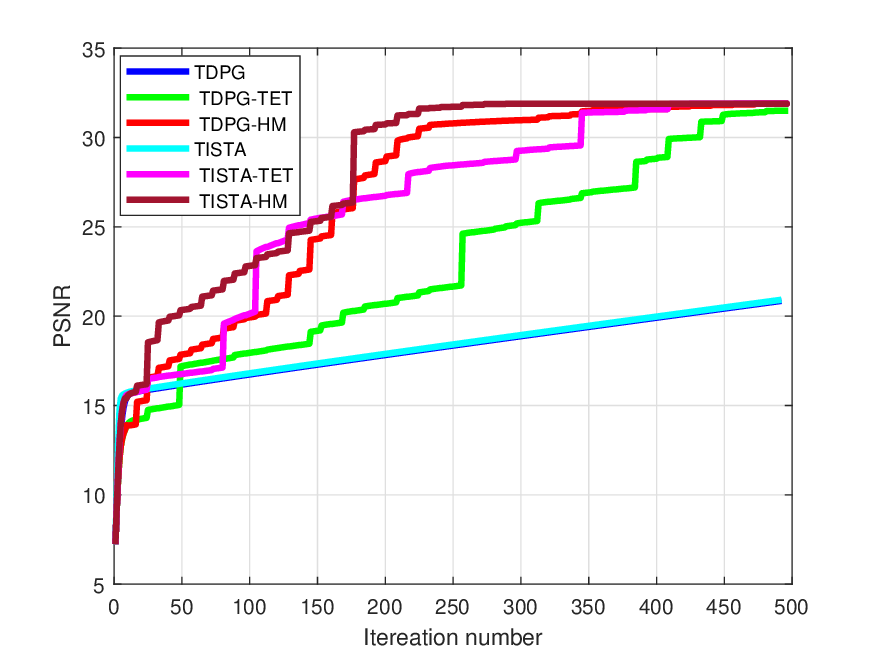} 
    \caption{ The curves of PSNR.}
    \label{figure3r1}  
      \end{minipage}\\
      \begin{center}
\begin{minipage}{5.8cm}
   \includegraphics[width=1.4\textwidth]{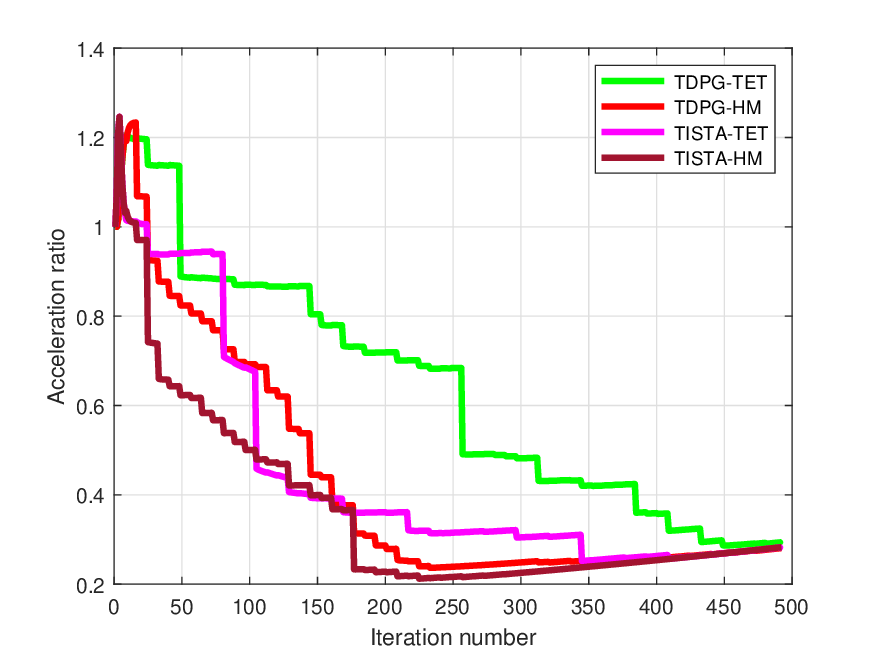}      
\caption{ The curves of the accelertion rate.}
\label{figure3r2}
 \end{minipage}
 \end{center}
 \end{figure}
\newpage
\subsection{Nonrandom missing pixels}

In this subsection, we choose the $250\times 250\times 3$  color images 'boats.png', 'hand.jpg' and 'soccer.jpg', and we added to these images different masks as shown in Figure \ref{Figures1}.
\begin{figure}[h]
      \centering
    \includegraphics[width=0.3\textwidth]{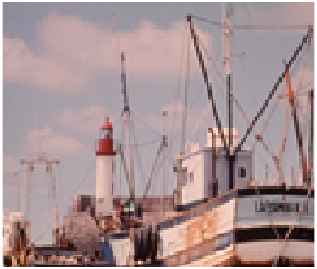}
    \includegraphics[width=0.3\textwidth]{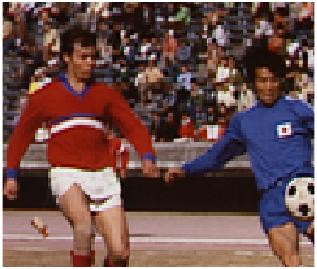}   
    \includegraphics[width=0.3\textwidth]{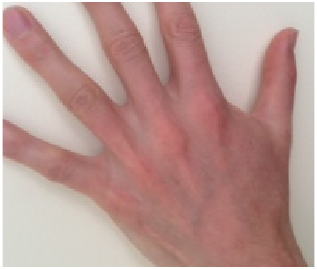}  \\    
        \centering
    \includegraphics[width=0.3\textwidth]{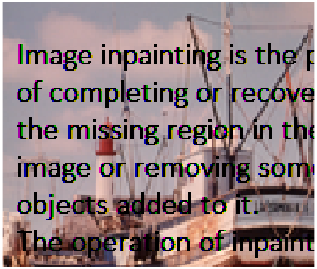}
    \includegraphics[width=0.3\textwidth]{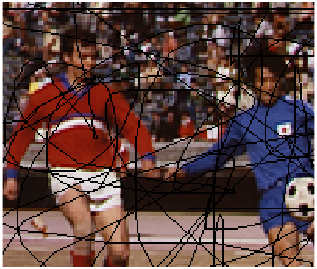}   
    \includegraphics[width=0.3\textwidth]{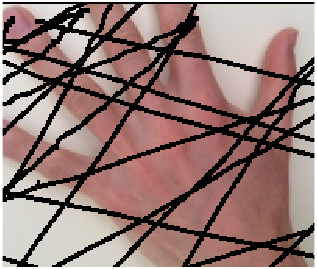}      

    \caption{The original images (1st row) and  the uniformly incomplete images (2nd row). }\label{Figures1}
\end{figure}
\noindent The recovered images are shown in Figure \ref{figur2s}. As in the  random case, we remark that TISTA and TDPG have nearly the same performance, while their extrapolated versions  provide more clear images.\\
\noindent The relative error and the PSNR curves of the six algorithms are  shown   in   Figures \ref{figure3s} and \ref{figure3s1}, respectively. The curves behaviours  show the applicability and effectiveness of the extrapolated algorithms TISTA-TET, TISTA-HM,  TDPG-TET and TDPG-HM. \\
 Figure \ref{figure3s2} illustrates the convergence  rate curves of  all extrapolated methods. We can see clearly the fast convergence of TISTA-TET and TDPG-TET as compated to  TISTA-HM and  TDPG-HM. This shows the advantage of  the extrapolated method using the topological epsilon algorithm for nonrandom incompletion problems.

\begin{figure}[htbp]
\begin{minipage}{1.6cm}
        \centering
        \caption*{ \ssmall TISTA}
    \includegraphics[width=1.3\textwidth]{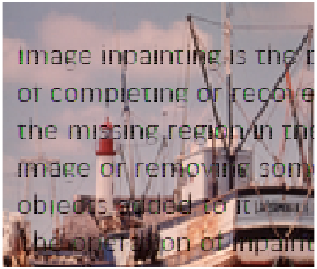}
\end{minipage}\hfill
\begin{minipage}{1.6cm}
        \centering
        \caption*{\ssmall TISTA-TET}
    \includegraphics[width=1.3\textwidth]{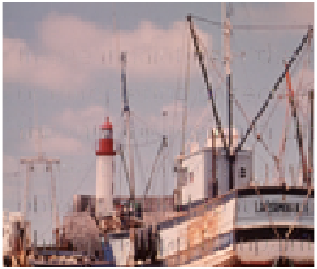}   
\end{minipage}\hfill
\begin{minipage}{1.6cm}
        \centering
        \caption*{ \ssmall TISTA-HM}
    \includegraphics[width=1.3\textwidth]{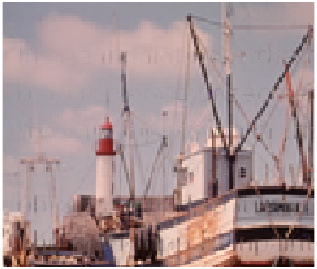}
\end{minipage}\hfill
\begin{minipage}{1.6cm}
        \centering
        \caption*{\ssmall TDPG}
    \includegraphics[width=1.3\textwidth]{boats_Reco_TDPG.png}    \end{minipage}\hfill
    \begin{minipage}{1.6cm}
        \centering
        \caption*{\ssmall TDPG-TET}
    \includegraphics[width=1.3\textwidth]{boats_Reco_HOSVD.png}        \end{minipage}\hfill
    \begin{minipage}{1.6cm}
        \centering
        \caption*{\ssmall TDPG-HM}
    \includegraphics[width=1.3\textwidth]{Boats_Reco_TET.png}
    \end{minipage}\\
\centering
\begin{minipage}{1.6cm}
        \centering
    \includegraphics[width=1.3\textwidth]{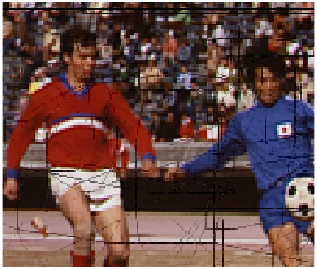}      
\end{minipage}\hfill
    \begin{minipage}{1.6cm}
        \centering
    \includegraphics[width=1.3\textwidth]{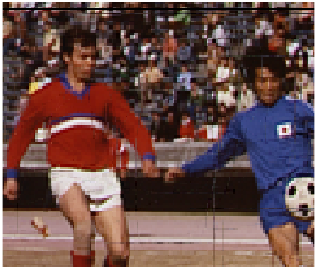}       
    \end{minipage}\hfill
    \begin{minipage}{1.6cm}
        \centering
    \includegraphics[width=1.3\textwidth]{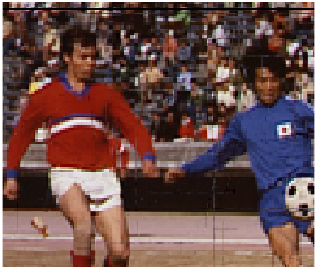}
    \end{minipage}\hfill
\begin{minipage}{1.6cm}
        \centering
    \includegraphics[width=1.3\textwidth]{soccer_Reco_TDPG.png}      
\end{minipage}\hfill
    \begin{minipage}{1.6cm}
        \centering
    \includegraphics[width=1.3\textwidth]{soccer_Reco_TET.png}       
    \end{minipage}\hfill
    \begin{minipage}{1.6cm}
        \centering
    \includegraphics[width=1.3\textwidth]{soccer_Reco_HOSVD.png}
    \end{minipage}\\
    
\begin{minipage}{1.6cm}
        \centering       \includegraphics[width=1.3\textwidth]{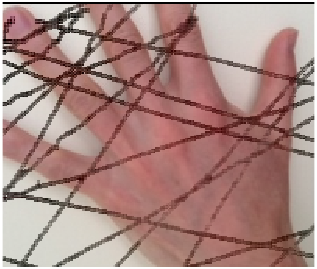}      
    \end{minipage}\hfill
    \begin{minipage}{1.6cm}
        \centering       \includegraphics[width=1.3\textwidth]{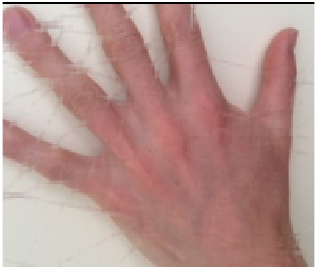}       
    \end{minipage}\hfill
    \begin{minipage}{1.6cm}
        \centering       \includegraphics[width=1.3\textwidth]{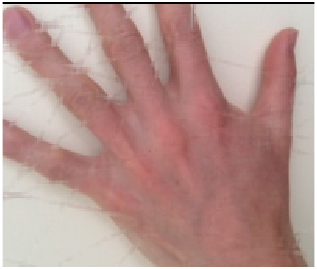}
    \end{minipage} \hfill
    \begin{minipage}{1.6cm}
        \centering       \includegraphics[width=1.3\textwidth]{hand_Reco_TDPG.png}      
    \end{minipage}\hfill
    \begin{minipage}{1.6cm}
        \centering       \includegraphics[width=1.3\textwidth]{hand_Reco_TET.png}       
    \end{minipage}\hfill
    \begin{minipage}{1.6cm}
        \centering       \includegraphics[width=1.3\textwidth]{hand_Reco_HOSVD.png}
    \end{minipage} 
    \caption{The original images (1st column), the corrupted images (2nd column), the recovered images by TDPG (3th column), by TDPG-TET (4th column) and by  TDPG-HM (5th column).}\label{figur2s}
\end{figure}

\begin{table}
\caption{Comparison of   algorithms  for  different incompletion levels $30\%$, $50\%$ \text{and}  $80\%$.}\label{T1} 
\centering
\begin{tabular}{c|c|c|c|c|c}
  \hline
   Images & $\%$ & Algorithms &  PSNR & Relative error & CPU-time (s)\\ \hline
  \multirow{9}{*}{\rotatebox{-60}{football.png}}  &\multirow{3}{*}{$30\%$} &TDPG/TISTA & 34.69/34.69& 4.25e-02/4.24e-02 & 12.93/12.14  \\
    & &TDPG-TET/TISTA-TET & 36,59/36.67  &  3.45e-02/ 3.42e-02 & 6.63 /5.72  \\
    & &TDPG-HM/TISTA-HM & 36.52/36.54  & 3.48e-02/3.47e-02 & 3.56/4.86  \\
    \cline{2-6}
  &\multirow{3}{*}{$50\%$} & TDPG/TISTA  & 28.21/28.21 &  9.08e-02/9.08e-02 & 19.83/18.59 \\
    & &TDPG-TET/TISTA-TET & 29.48/29.51 & 7.81e-02/7.79e-02 & 10.94/7.05 \\
    & & TDPG-HM/TISTA-HM &29.52/29.54 & 7.78e-02/7.76e-02 & 10.09/7.88 \\    
     \cline{2-6}
  &\multirow{3}{*}{$80\%$} &TDPG/TISTA&  19.10/19.12 &  2.83e-01/2.82e-01 & 35.10/39.32 \\
    & &TDPG-TET/TISTA-TET & 22,76/22.82 & 1.71e-01/ 1.70e-01& 20.36/14.23 \\
    & & TDPG-HM/TISTA-HM & 22.72/22.69 & 1.72e-01/ 1.73e-01 & 24.75/11.89 \\
     \hline 
  \multirow{9}{*}{\rotatebox{-60}{barbara.bmp}} &\multirow{3}{*}{$30\%$} & TDPG/TISTA  & 36.11/36.11 & 3.38e-02/3.38e-02 & 12.77/12.78  \\
    & & TDPG-TET/TISTA-TET & 37.83 /87.87 &  2.75e-02/2.74e-02 & 6.82/4.4   \\
    & &TDPG-HM/TISTA-HM& 37.77/37.78  & 2.77e-02/ 2.77e-02 & 10.12/5.41  \\
   \cline{2-6}
  &\multirow{3}{*}{$50\%$} &TDPG/TISTA &29.92/29.92 &  7.00e-02/7.00e-02 & 32.08/37.08 \\
    & &TDPG-TET/TISTA-TET & 31.39/31.61 & 5.82e-02/5.67e-02 & 28.39/33.69 \\
    & & TDPG-HM/TISTA-HM & 31.40/31.40 &  5.82e-02/ 5.82e-02 & 19.44/10.86 \\    
     \cline{2-6}
  &\multirow{3}{*}{$80\%$} &TDPG/TISTA  & 21.44/21.44 &  1.92e-01/1.92e-01 & 45.54/48.67 \\
    & &TDPG-TET/TISTA-TET & 23.17/23.07 & 1.54e-01/1.56e-01 & 35.96/27.69 \\
    & & TDPG-HM/TISTA-HM & 23.21/23.20 & 1.53e-01/1.53e-01 & 34.94/23.32 \\
     \hline
  \multirow{9}{*}{\rotatebox{-60}{flamingos.jpg}} 
  &\multirow{3}{*}{$30\%$} &TDPG/TISTA  & 34.93/34.92 &4.51e-02/4.51e-02 & 12.82/13.93   \\
    & &TDPG-TET/TISTA-TET & 36.61/36.85  &  3.63e-02/3.54e-02& 6.98/7.88   \\
    & &TDPG-HM/TISTA-HM & 36.60/36.60  & 3.64e-02/3.64e-02 & 11.95/6.84  \\
    \cline{2-6}
  &\multirow{3}{*}{$50\%$} &TDPG/TISTA & 28.29/28.29 &  9.92e-02/ 9.91e-02 & 22.45/25.30 \\
    & &TDPG-TET/TISTA-TET & 29.66/29.60 & 8.28e-02/8.35e-02 & 10.62/9.26 \\
    & & TDPG-HM/TISTA-HM & 29.62/29.63 & 8.28e-02/8.33e-02 & 10.15/9.49 \\    
     \cline{2-6}
  &\multirow{3}{*}{$80\%$} &TDPG/TISTA & 20.49/20.48 & 2.57e-012.56e-01 & 43.35 /45.01\\
    & &TDPG-TET/TISTA-TET & 21.80/21.96 & 2.20e-01/2.16e-01 & 24.01/16.95 \\
    & &TDPG-HM/TISTA-HM & 21.97/21.96 & 2.15e-01/2.14e-01 & 22.49/1854 \\
     
     \hline 
\end{tabular}
\end{table}

\section{Conclusion}\label{S7}
In this work, by embracing a proximal-based approach, we introduced a solution for a constrained multidimensional  minimization problem with a general non-smooth regularization term. First, we introduce an approximation of the unconstrained problem (in the whole space). An approximated  solution of the constrained problem is obtained by projecting, via Tseng's Algorithm, the unconstrained solution into  the closed set $\Omega$. Through  numerical experiments on image completion, we have illustrated the efficiency of the adopted approach. Due to the slowness of the proposed algorithms, we have enhancing their convergence rate by incorporating some recent tensor extrapolation methods. The suggest numerical tests illustrate the significant role of these extrapolation methods in speeding up the convergence of the proposed algorithms.

\begin{figure}[ht]
 \begin{minipage}{5.8cm}
    \includegraphics[width=1.4\textwidth]{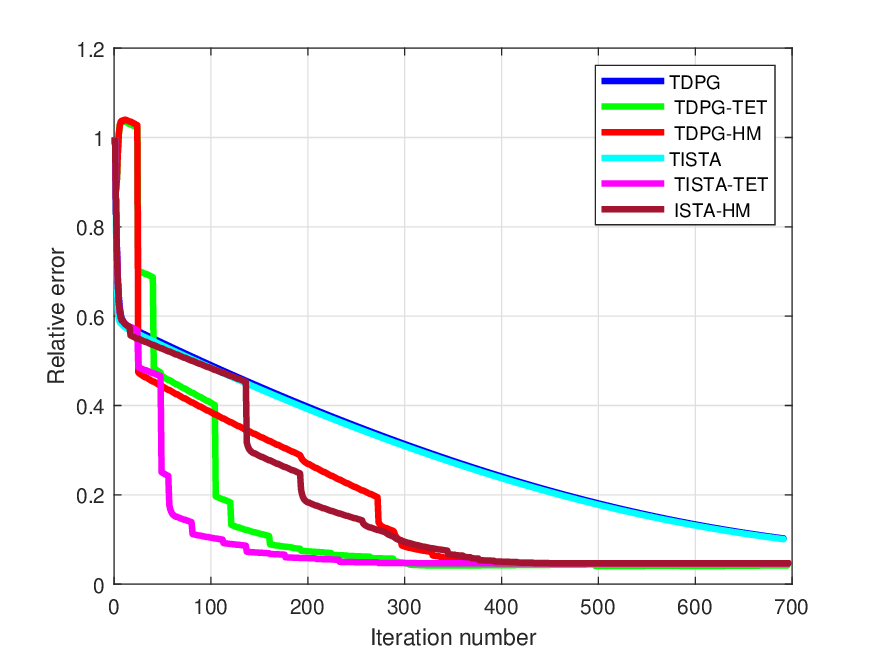}
    \caption{ The curves of relative errorr.}
    \label{figure3s}
     \end{minipage}\hspace{1.7cm}
\begin{minipage}{5.8cm}
    \includegraphics[width=1.4\textwidth]{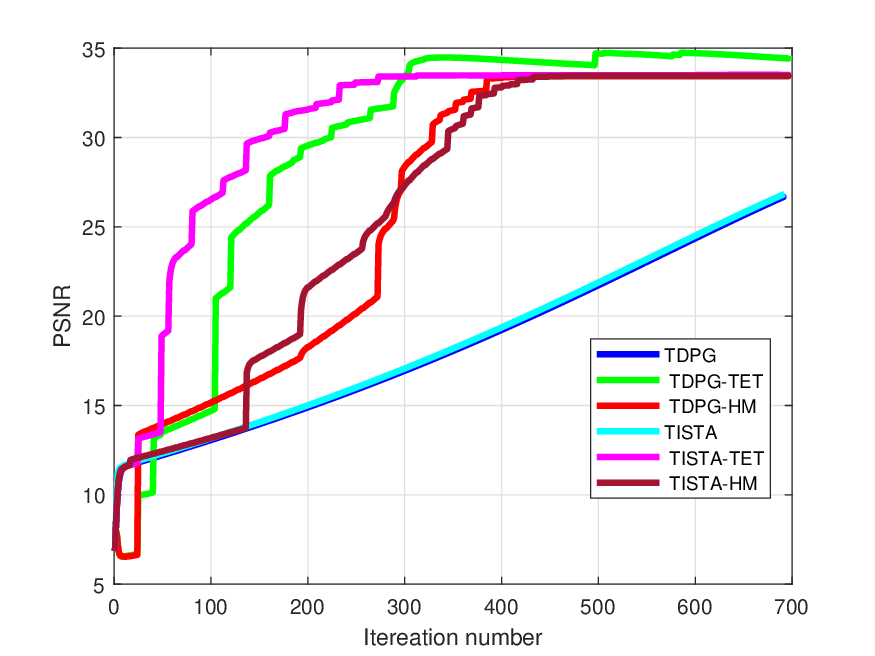} 
    \caption{ The curves of PSNR.}
    \label{figure3s1}  
      \end{minipage}\\
      \begin{center}
\begin{minipage}{5.8cm}
    \includegraphics[width=1.4\textwidth]{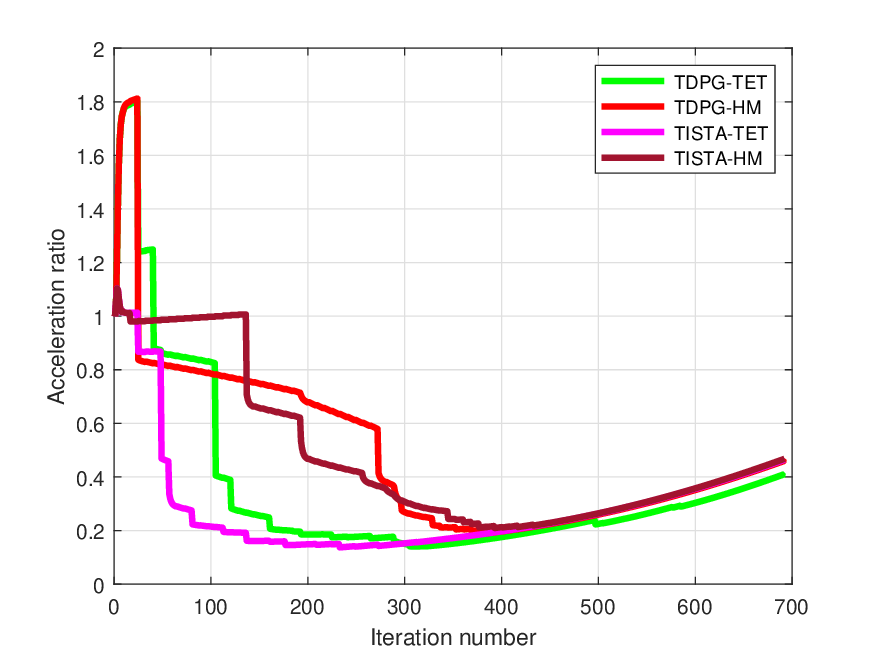}      
\caption{ The curves of the accelertion rate.}
\label{figure3s2}
 \end{minipage}
 \end{center}
 \end{figure}


\end{document}